\documentclass[12pt, twoside]{article}
\usepackage{amsmath,amsthm,amssymb}
\usepackage{times}
\usepackage{enumerate}

\def\titlerunning#1{\gdef\titrun{#1}}
\makeatletter
\def\author#1{\gdef\autrun{\def\and{\unskip, }#1}\gdef\@author{#1}}
\def\address#1{{\def\and{\\\hspace*{18pt}}\renewcommand{\thefootnote}{}%
\footnote {#1}}%
\markboth{\autrun}{\titrun}}
\makeatother
\def\email#1{e-mail: #1}
\def\subjclass#1{{\renewcommand{\thefootnote}{}%
\footnote{\emph{Mathematics Subject Classification (2010):} #1}}}

\newtheorem{thm}{Theorem}[section]

\theoremstyle{definition}
\newtheorem{defin}[thm]{Definition}

\numberwithin{equation}{section}

\newcommand{\C}{{\mathbb C}}

\newcommand{\cone}{\operatorname{cone}}

\newcommand{\const}{\operatorname{const.}}

\newcommand{\cyl}{\operatorname{cyl}}

\newcommand{\Hess}{\operatorname{Hess}}

\newcommand{\HH}{\operatorname{H}}

\newcommand{\Id}{\operatorname{Id}}

\newcommand{\Image}{\operatorname{Im}}

\newcommand{\Int}{\operatorname{int}}

\newcommand{\R}{{\mathbb R}}

\newcommand{\re}{\operatorname{Re}}

\newcommand{\Ric}{\operatorname{Ric}}
\newcommand{\Rm}{\operatorname{Rm}}

\newcommand{\sol}{\operatorname{sol}}

\newcommand{\supp}{\operatorname{supp}}

\newcommand{\Tr}{\operatorname{Tr}}

\newcommand{\Z}{{\mathbb Z}}

\numberwithin{equation}{section}
\theoremstyle{plain}

\newtheorem{lemma}[equation]{Lemma}

\newtheorem{proposition}[equation]{Proposition}
\newtheorem{corollary}[equation]{Corollary}

\theoremstyle{definition}

\theoremstyle{definition}
\newtheorem{example}[equation]{Example}

\theoremstyle{definition}
\newtheorem{remark}[equation]{Remark}
\def\<{\langle}
\def\>{\rangle}
\def\({\left(}
\def\){\right)}
\def\p{\partial}

\begin{document}


\baselineskip=17pt


\titlerunning{Ricci flow on quasiprojective manifolds II}

\title{Ricci flow on quasiprojective manifolds II}

\author{John Lott
\and 
Zhou Zhang}

\date{}

\maketitle

\address{John Lott: Department of Mathematics,
University of California - Berkeley,
Berkeley, CA  94720-3840,
USA; \email{lott@berkeley.edu}
\and
Zhou Zhang: School of Mathematics and Statistics,
The University of Sydney,
NSW 2006, Australia; \email{zhangou@maths.usyd.edu.au}}



\begin{abstract}
We study the Ricci flow on complete K\"ahler metrics that live on
the complement of a divisor in a compact complex manifold. In earlier work,
we considered finite-volume metrics which, at spatial infinity, are
transversely hyperbolic.  In the present paper we consider 
three different types of spatial asymptotics: cylindrical, bulging and
conical.  We show that in each case, the asymptotics are preserved
by the K\"ahler-Ricci flow.  We address long-time existence,
parabolic blowdown limits
and the role of the K\"ahler-Ricci flow on the divisor.
\end{abstract}





\section{Introduction} \label{section1}

Let $\overline{X}$ be a compact complex manifold of complex
dimension $n$. Let $D$ be a divisor
in $\overline{X}$ with normal crossings.
In a series of papers, Tian and Yau gave sufficient conditions for the
quasiprojective manifold
$X = \overline{X} - D$ to admit a complete K\"ahler-Einstein metric.
The papers differed in the kind of spatial asymptotics that were
considered.  The paper \cite{tian-yau} gave sufficient conditions for
$X$ to admit a finite-volume K\"ahler-Einstein metric with negative
Ricci tensor, having cuspidal asymptotics. 
The papers \cite{tian-yau2} and \cite{tian-yau3}
dealt with Ricci-flat K\"ahler metrics. The paper
\cite{tian-yau2} considered two types of spatial asymptotics,
which we call cylindrical and bulging.  The paper
\cite{tian-yau3} considered asymptotically conical metrics.

In \cite{lz-duke}, we looked at the K\"ahler-Ricci flow on
quasiprojective manifolds
with cuspidal asymptotics. That is, the initial
K\"ahler metric on $X$ was assumed to be 
complete, finite-volume and asymptotic to a hyperbolic cusp in directions 
transverse to $D$.
Unlike in \cite{tian-yau}, no assumptions were made on the 
divisor class, since we were not necessarily looking for
K\"ahler-Einstein metrics. We defined two flavors of spatial asymptotics:
``standard'' spatial asymptotics, which describes the leading
behavior of the K\"ahler form at spatial infinity, and
``superstandard'' spatial asymptotics, in which the leading behavior
of the K\"ahler
potential is also prescribed.
We showed that standard and superstandard spatial asymptotics
are both preserved by the K\"ahler-Ricci flow (taking into account the
K\"ahler-Ricci flow on $D$).  We 
gave a formula for the first singularity time, if there is one,
in terms of the cohomology of $\overline{X}$.

In the present paper we look at the technically more
challenging case of the K\"ahler-Ricci flow on quasiprojective
manifolds with spatial asymptotics similar to those considered
in \cite{tian-yau2} and \cite{tian-yau3}. As in
our earlier paper, the goal is to define classes of spatial
asymptotics that are as general as possible while being preserved under the
K\"ahler-Ricci flow, and about which one can say something nontrivial.
We consider ``cylindrical'', ``bulging'' and ``conical''
spatial asymptotics.  In each case, there are notions of
``standard'' asymptotics and ``superstandard'' asymptotics.

\begin{thm} \label{intro1}
If the initial K\"ahler metric has  (cylindrical, bulging
or conical) (standard or superstandard)
spatial asymptotics then so do the evolving K\"ahler metrics 
of the K\"ahler-Ricci flow.
With superstandard asymptotics,
\begin{enumerate} [\upshape (i)]
\item 
In the cylindrical or bulging case, if $K_{\overline{X}} + D \ge 0$ in 
$H^{(1,1)}(\overline{X}; \R)$ then the K\"ahler-Ricci flow
exists for all positive time.
\item
In the conical case, if $K_{\overline{X}} + (n+1) D \ge 0$ in 
$H^{(1,1)}(\overline{X}; \R)$ then the K\"ahler-Ricci flow
exists for all positive time.
\end{enumerate}
\end{thm}

One theme of this paper is the relationship between the
K\"ahler-Ricci flow on $X$ and the K\"ahler-Ricci flow on the
divisor $D$. The result depends on what kind of spatial asymptotics
we are considering, so we address them separately. 
\subsection{Cylindrical spatial asymptotics} (c.f. \cite[Section 5]{tian-yau2}) 
In this case, each
component $\{D_i\}_{i=1}^k$ of $D$ is assumed to have a trivial normal
bundle.  More precisely, we assume that there are holomorphic fiberings
$f_i : \overline{X} \rightarrow C_i$ from $\overline{X}$ 
to complex curves so that
$D_i = f^{-1}(s_i)$ for some $s_i \in C_i$.
If $z_i$ is a local coordinate for $C_i$ near $s_i$, let 
$z_i$ also denote its pullback to $\overline{X}$.
If $I = (i_1, \ldots, i_m), 1 \le i_1 < \ldots < i_m \le k$  is an
ordered set, put 
$D_I = D_{i_1} \cap \ldots \cap D_{i_m}$. Let $D_I^{\Int}$ be the
smooth manifold consisting of
points in $D_I$ that do not lie in lower-dimensional strata.
Let $\omega_{D_I^{\Int}}$ be a complete K\"ahler metric on
$D_I^{\Int}$.

Given $\overline{x} \in D$, let $I$ be such that
$\overline{x} \in D_I$ but
$\overline{x} \notin D_i$ for $i \notin I$.
If a K\"ahler metric $\omega_X$ on $X$
has cylindrical standard spatial asymptotics, associated to positive numbers
$\{ c_i \}_{i=1}^k$ and the metrics 
$\{\omega_{D_I^{\Int}} \}$, then for all $\overline{x} \in D$, the
asymptotics of $\omega_X$ ``near'' $\overline{x}$ are
\begin{equation}
\omega_X \sim \sum_{i \in I} 2c_i \sqrt{-1} \frac{dz^i \wedge
d\overline{z}^i}{|z_i|^2}
+ \omega_{D_I^{\Int}}.
\end{equation}
That is, in the way of approaching spatial infinity specified by 
$\overline{x}$,
the metric $\omega_X$ looks like a product of Euclidean cylinders with
$\omega_{D_I^{\Int}}$. The K\"ahler-Ricci flow on the divisor enters into
the K\"ahler-Ricci flow on $X$ in the following way.

\begin{thm} \label{intro2}
Suppose that the initial metric 
$\omega_X(0)$ has cylindrical standard spatial asymptotics associated to
$\{c_i\}_{i=1}^k$ and $\{\omega_{D_I^{\Int}}(0) \}$.
Then under the K\"ahler-Ricci flow on $X$, the metric
$\omega_X(t)$ has standard spatial asymptotics associated to
$\{c_i\}_{i=1}^k$ and $\{\omega_{D_I^{\Int}}(t) \}$, where
$\omega_{D_I^{\Int}}(\cdot)$ is the K\"ahler-Ricci flow on
$D_I^{\Int}$ starting from $\omega_{D_I^{\Int}}(0)$.
\end{thm}

\subsection{ Bulging spatial asymptotics} (c.f. \cite[Section 4]{tian-yau2})
Suppose that $D$ is a smooth divisor, connected for simplicity. 
Let $x_0$ be a basepoint
in $X = \overline{X} - D$. Suppose that $n > 1$
and let $\omega_D$ be a K\"ahler metric on
$D$. If $X$ has bulging standard spatial
asymptotics, associated to $\omega_D$, then a sphere of large distance
$R$ from $x_0$ is the total space of a circle bundle over $D$.
As $R \rightarrow \infty$, the lengths of the circle fibers are 
$O \left( R^{\: - \: \frac{n-1}{n+1}} \right)$.
The metric on the base of the bundle is comparable to
$R^{\frac{2}{n+1}} \omega_D$. The sectional curvatures are
$O \left( R^{\: - \: \frac{2}{n+1}} \right)$.

If $\omega_X(0)$ has bulging standard spatial asymptotics
associated to $\omega_D(0)$, and
$\omega_X(\cdot)$ is the ensuing K\"ahler-Ricci flow, then it
turns out that $\omega_X(t)$ has  
bulging standard spatial asymnptotics
associated to $\omega_D(0)$. That is, the divisor flow does not
enter into the finite-time spatial asymptotics on $X$.
The proof of this uses results from Appendix \ref{powerapp}
about the preservation, under K\"ahler-Ricci flow, of a power law
decay in the curvature. 
Intuitively, in order to see a significant change in
the geometry at a point $x \in X$, there must be an elapsed time
comparable to $d_0(x,x_0)^{\frac{2}{n+1}}$, where $d_0$ is
the time-zero distance. 

This suggests 
that to see the divisor flow, one should take
a sequence of time-zero points $\{x_i\}_{i=1}^\infty$ in $X$, going to 
spatial infinity, and perform a parabolic rescaling around
$(x_i, 0)$ by a factor of $d_0(x_i,x_0)^{- \: \frac{2}{n+1}}$, i.e.
\begin{equation} \label{intro1.4}
\omega_i(t) = d_0(x_i,x_0)^{- \: \frac{2}{n+1}} \:
\omega_X \left( d_0(x_i,x_0)^{\frac{2}{n+1}} t \right).
\end{equation}
The ensuing pointed limit should see the divisor flow.

There is a technical issue in taking such a blowdown limit.
Namely, one needs uniform sectional curvature bounds on
forward parabolic balls
in the rescaled metrics.  Such curvature bounds would follow from
Perelman's pseudolocality result
\cite{Lu,Perelman}, if it were applicable.  
Unfortunately, pseudolocality does not apply in this case,
because of the shrinking circle fibers.

To get uniform sectional curvature bounds, we 
instead use a remarkable feature of
K\"ahler-Ricci flow: a local biLipschitz bound implies a local
curvature bound, independent of the elapsed time.
The proof of this statement is given in Appendix \ref{localapp},
building on work of Sherman-Weinkove \cite{sherman-weinkove}.
To get biLipschitz estimates, it suffices to have a Ricci curvature
bound.

\begin{thm} \label{intro3}
Suppose that the initial metric $\omega_X(0)$ has bulging standard
spatial asymptotics associated to $\omega_D(0)$. Let
$(X, \omega_X(\cdot))$ be the K\"ahler-Ricci flow starting from
$\omega_X(0)$. \\
(a) For any time $t \ge 0$, the K\"ahler metric $\omega_X(t)$ also
has bulging standard
spatial asymptotics associated to $\omega_D(0)$. \\
(b) Suppose that the flow $(X, \omega_X(\cdot))$ exists
for all positive time. Given $A < \infty$, 
suppose that $|\Ric(x,t)| = 
O \left( d_0(x,x_0)^{\: - \: \frac{2}{n+1}} \right)$
on the spacetime region $\left\{ (x,t) \: : \: t \: \le \: A \: 
d_0(x,x_0)^{\frac{2}{n+1}} \right\}$.
Let $\{x_i\}_{i=1}^\infty$ be a sequence
of points in $X$ going to spatial infinity.  Define $\omega_i(\cdot)$
by (\ref{intro1.4}). Then the pointed Ricci flow limit 
$\lim_{i \rightarrow \infty}
\left( X, (x_i, 0), \omega_i(\cdot) \right) = 
\left( X_\infty, (x_\infty, 0), \omega_\infty(\cdot) \right)$ 
exists on the time interval $[0,A]$ and is given by
\begin{equation} \label{intro1.6}
\omega_\infty(t) = \frac12 \left( \frac{n+1}{n} \right)^2
\left( \sqrt{-1} du \wedge d\overline{u} + n \omega_D(t) \right).
\end{equation} 
\end{thm}
Here $X_\infty$ is interpreted as an \'etale groupoid of
complex dimension $n$ (because of the collapsing
circle fibers) with unit space $\C \times D$, so $u$ lies in $\C$.
Alternatively, one can express the convergence in terms of local
covers; see Proposition \ref{5.8}.

\subsection{Conical asymptotics} (c.f. \cite{tian-yau3}) Suppose that $D$ is a
smooth divisor, connected for simplicity. Let $\omega_D$ be a
K\"ahler metric on $D$.  If $X$ has conical standard spatial
asymptotics, associated to $\omega_D$, then it has quadratic curvature decay. 
The asymptotic cone of $X$ is a metric cone
$(CY, \omega_{CY})$, 
where $Y$ is the total space of a circle bundle over $D$,
the latter having metric $\omega_D$. More precisely, the conical metric
$\omega_{CY}$ is defined on $CY - \star_{CY}$, where
$\star_{CY} \in CY$ is the vertex.

If $\omega_X(0)$ has conical standard spatial asymptotics, associated
to $\omega_D$, then $\omega_X(t)$ also has conical standard spatial 
asymptotics associated to $\omega_D$. To go beyond this, it is
natural to look at parabolic blowdowns. Suppose
that the flow $(X, \omega_X(\cdot))$ exists for all positive time.
Using pseudolocality, we show that one can extract a blowdown 
Ricci flow limit
$\omega_\infty(\cdot)$ that exists on
$\left\{ (x, t) \in (CY - \star_{CY}) \times [0, \infty) \: : \:
0 \le t \le \epsilon d_{CY}(x, \star_{CY})^2 \right\}$, for some
$\epsilon > 0$. (One can also form such blowdown limits in the
nonK\"ahler case.) Its initial metric $\omega_{\infty}(0)$ is
the conical metric $\omega_{CY}$.

One may hope to see dynamics on $D$ entering into the asymptotic cone
of the blowdown limit. 
However, one finds that the asymptotic cone of a time slice 
$\left( CY - \star_{CY}, \omega_{\infty}(t) \right)$ of the blowdown limit 
is still the asymptotic cone $(CY, \omega_{CY})$ of $(X, \omega_X(0))$,
constructed using the initial metric $\omega_D$ on $D$.
In this conical setting, the dynamical object associated
to $D$ is no longer a K\"ahler-Ricci flow on $D$.
Rather, we show that there is a formal gradient expanding soliton
$\omega_{\sol}(\cdot)$ on $CY - \star_{CY}$, which is
uniquely determined by $\omega_D$. The time-one slice of this soliton
takes the form
\begin{equation}
\omega_{\sol}(1) = \omega_{CY} - \Ric(\omega_{CY}) + \sqrt{-1}
\partial \overline{\partial} \sum_{k > 0} u_{(k)}.
\end{equation}
The summation is a formal power series in a
holomorphic coordinate $z$ transverse to $D$ 
(and its complex conjugate) or, equivalently,
in inverse powers of the distance to $\star_{CY}$. 
The vector field associated to the soliton is the radial vector field
on $CY$.

\begin{thm} \label{intro4}
Suppose that the initial metric $\omega_X(0)$ has conical standard
spatial asymptotics associated to $\omega_D$. Let  $(X, \omega_X(\cdot))$
be the K\"ahler-Ricci flow starting from $\omega_X(0)$. \\
(a) For any time $t \ge 0$, the K\"ahler metric $\omega_X(t)$ has
conical standard spatial asymptotics associated to $\omega_D$. \\ 
(b) Suppose that the flow $(X, \omega_X(\cdot))$ exists for all positive
time. Let $\omega_\infty(\cdot)$ be a parabolic blowdown limit.
Then any asymptotic expansion of $\omega_\infty(\cdot)$, in
powers of $z$ and $\overline{z}$, equals the gradient expanding soliton 
$\omega_{\sol}(\cdot)$ associated to $\omega_D$.   
\end{thm}

\subsection{Structure of the paper}

In Section \ref{section2} we recall some results from
\cite{lz-duke}. In Section \ref{section3} we define
cylindrical standard spatial asymptotics and
cylindrical superstandard spatial asymptotics. We show that
they are preserved by the K\"ahler-Ricci flow, when taking
into account the flow on the divisor. We give a formula for
the first singularity time, if there is one.

In Section \ref{section4} we define
bulging standard spatial asymptotics and
bulging superstandard spatial asymptotics. We show that
they are preserved by the K\"ahler-Ricci flow.
We give a formula for
the first singularity time, if there is one.
With a decay assumption on the Ricci curvature, we
show that there is a parabolic blowdown limit based
on a sequence of points on the time-zero slice which
go to spatial infinity.  We prove that the blowdown limit
is a product flow in which the divisor flow appears.

In Section \ref{section5} we define
conical standard spatial asymptotics and
conical superstandard spatial asymptotics. We show that
they are preserved by the K\"ahler-Ricci flow.
We give a formula for
the first singularity time, if there is one.
We show that one can take parabolic blowdowns of asymptotically
conical Ricci flows. We construct the formal asymptotic
expansion of a gradient expanding soliton on the 
asymptotic cone. We show that any formal asymptotic expansion
of the parabolic blowdown equals the expanding soliton.

There are two appendices which may be of independent interest.
In Appendix \ref{localapp} we prove a local curvature estimate
in K\"ahler-Ricci flow, showing that a biLipschitz bound implies
a curvature bound, independent of the elapsed time.
The proof is along the lines of Sherman-Weinkove \cite{sherman-weinkove}, with
some modifications.  The result of Appendix \ref{localapp}
is used in Section \ref{section4} and Appendix \ref{powerapp}.

In Appendix \ref{powerapp} we discuss the preservation, under
Ricci flow, of a power law decay of the curvature tensor.  For general
Ricci flow, we show that such a decay is preserved whenever the
initial geometry has a reasonable end structure.  More precisely, we need a
smooth distance-like function $\phi$ so that 
$\frac{\Hess(\phi)}{\phi}$ is bounded below. The proof is along
the lines of Dai-Ma \cite{Dai-Ma}. Using Appendix \ref{localapp}, we
show that a power law decay of the curvature is preserved under 
K\"ahler-Ricci flow, provided that the first derivatives of the
initial curvature tensor also have the right decay.
The results of Appendix \ref{powerapp} are used in
Sections \ref{section4} and \ref{section5}.

More detailed descriptions are given at the beginnings of the
sections.

\subsection{Further directions}

One can ask about more refined asymptotics for the K\"ahler-Ricci
flow in the cases being considered.  In the
cuspidal case discussed in \cite{lz-duke}, such
asymptotics were developed in Rochon-Zhang \cite{rz-advances}.

One can also ask about long-time behavior, at least when
$K_{\overline{X}} + D \ge 0$ in the cylindrical and bulging cases,
or when $K_{\overline{X}} + (n+1) D \ge 0$ in the conical case.
Chau \cite{chau} and the authors \cite[Theorem 5.1]{lz-duke} 
gave sufficient conditions in the noncompact case
to ensure convergence to a K\"ahler-Einstein metric of negative
Ricci curvature. For convergence to a 
noncompact Ricci-flat K\"ahler metric,
the only result that we know is by Chau and Tam
\cite{chau-tam}, who prove convergence to a Ricci-flat K\"ahler
metric if a Sobolev inequality holds for the initial metric 
(relevant to conical asymptotics)
and if the Ricci potential of the initial metric 
has a faster-than-quadratic decay.

\section{Background material} \label{section2}

Throughout this paper, we use the notation and conventions
of \cite[Section 2]{lz-duke}. In particular,
$\Delta$ is the open unit ball in $\C$ and $\Delta^* = \Delta - \{0\}$.
We let $H$ denote the upper half plane $\{ z \in \C \: : \:
\Image(z) > 0 \}$.
If $(Z, d_Z)$ is a metric space then for
$\lambda > 0$, we let $\frac{1}{\lambda} Z$ denote $Z$ with the
metric $\frac{d_Z}{\lambda}$. Thus if $(Z, g_Z)$ is a 
Riemannian manifold then $\frac{1}{\lambda} Z$ denotes
$Z$ with the Riemannian metric $\frac{g_Z}{\lambda^2}$.

Let $X$ be a connected complex manifold of complex dimension $n$.
Suppose that $\omega_0$ is a smooth complete K\"ahler metric on $X$
with bounded curvature. The K\"ahler-Ricci flow equation is
\begin{equation} \label{2.1}
\frac{\partial \widetilde{\omega}_t}{\partial t} = -
\Ric(\widetilde{\omega}_t), \: \: \: \: \: \: 
\widetilde{\omega}_0 = \omega_0.
\end{equation}
There is some $T > 0$ so that there is a solution of (\ref{2.1})
on the time interval $[0, T]$
having complete time slices and uniformly bounded
curvature on $[0,T]$. Furthermore, such a solution is unique.

Put 
\begin{equation} \label{2.2}
\omega_t = \omega_0 - t \Ric(\omega_0).
\end{equation}
Consider the equation 
\begin{equation} \label{2.3}
\frac{\partial u}{\partial t} = \log
\frac{(\omega_t + \sqrt{-1} \partial \overline{\partial} u)^n}{\omega_0^n}
\end{equation}
with the initial condition $u(0, \cdot) = 0$. It is implicit that we
only consider solutions $u$ of (\ref{2.3}) on time intervals so that
$\omega_t + \sqrt{-1} \partial \overline{\partial} u > 0$. The following results 
are from \cite{lz-duke}.

\begin{proposition} \label{2.4}
Suppose that there is a smooth solution to (\ref{2.1}) on a time interval
$[0, T]$, with complete time slices and uniformly bounded curvature.
Then there is a smooth solution for $u$ in (\ref{2.3}) on the time
interval $[0,T]$ so that 
\begin{enumerate} [\upshape (i)]
\item For each $t \in [0, T]$, $\omega_t + \sqrt{-1} \partial
\overline{\partial} u$ is a K\"ahler metric which is biLipschitz-equivalent
to $\omega_0$, and
\item For each $k$, the $k$th covariant derivatives of $u$
(with respect to the initial metric $\omega_0$) are uniformly bounded.
\end{enumerate}
Also, $\widetilde{\omega}_t = \omega_t + \sqrt{-1} \partial
\overline{\partial} u$.

Conversely, suppose that there is a smooth solution to (\ref{2.3}) on a
time interval $[0, T]$ so that
\begin{enumerate} [\upshape (i)]
\item For each $t \in [0, T]$, $\omega_t + \sqrt{-1} \partial
\overline{\partial} u$ is a K\"ahler metric which is 
biLipschitz-equivalent to $\omega_0$, and
\item For each $k$, the $k$th covariant derivatives of $u$
(with respect to the initial metric $\omega_0$) are uniformly bounded.
\end{enumerate}
Then $\widetilde{\omega}_t = \omega_t + \sqrt{-1} \partial
\overline{\partial} u$ is a solution to (\ref{2.1}) on $[0, T]$, 
with complete time slices and uniformly bounded curvature.
\end{proposition}

\begin{thm} \label{2.5}
Suppose that $\omega_0$ is a complete K\"ahler metric on a complex manifold
$X$, with bounded curvature.

Let $T_1$ be the supremum (possibly infinite) of the numbers $T^\prime \ge 0$
so that there is a smooth solution for (\ref{2.3}) on the time interval
$[0, T^\prime]$ such that
\begin{enumerate} [\upshape (i)]
\item For each $t \in [0, T^\prime]$, $\omega_t + \sqrt{-1}
\partial \overline{\partial} u$ is a K\"ahler metric which is
biLipschitz-equivalent to $\omega_0$, and 
\item For each $k$, the $k$th covariant derivatives of $u$
(with respect to the initial metric $\omega_0$) are uniformly
bounded on $[0, T^\prime]$.
\end{enumerate}

Let $T_2$ be the supremum (possibly infinite) of the numbers $T \ge 0$ for
which there is a function $F_T \in C^\infty(X)$ such that
\begin{enumerate} [\upshape (i)]
\item $\omega_T + \sqrt{-1} \partial \overline{\partial} F_T$ is a
K\"ahler metric which is biLipschitz-equivalent to $\omega_0$, and
\item For each $k$, the $k$th covariant derivatives of $F_T$ 
(with respect to the initial metric $\omega_0$) are uniformly bounded.
\end{enumerate}

Then $T_1 = T_2$.
\end{thm}

\section{Cylindrical K\"ahler-Ricci flows} \label{section3}

This section deals with cylindrical spatial asymptotics. In
Subsection \ref{subsection3.1} we define the notion of cylindrical standard
spatial asymptotics, and show that it is preserved under the
K\"ahler-Ricci flow, taking into account the divisor flow.
In Subsection \ref{subsection3.2} we introduce cylindrical superstandard
spatial asymptotics. We show that if $\overline{X}$ admits a
K\"ahler metric then $X$ admits a metric with
cylindrical superstandard spatial asymptotics.
We prove that having cylindrical superstandard spatial
asymptotics is preserved under the K\"ahler-Ricci flow. Given
a metric with cylindrical superstandard spatial asymptotics,
we define a certain renormalized cohomology class on the
compactification $\overline{X}$. We use this cohomology class
to characterize the first singularity time, if there is one.

Let $\overline{X}$ be a compact connected $n$-dimensional complex manifold.
For $1 \le i \le k$, let $f_i : \overline{X} \rightarrow C_i$ be a
holomorphic fibering over a complex curve $C_i$. 
Given points $s_i \in C_i$, let $g_{s_i} : \Delta \rightarrow C_i$ be a
local parametrization of $C_i$ with $g_{s_i}(0) = s_i$.
Given an ordered set
$I = (i_1, \ldots, i_m), 1 \le i_1 < \ldots < i_m \le k$, put
\begin{enumerate} [\upshape (i)]
\item $C_I = C_{i_1} \times \ldots \times C_{i_m}$,
\item $s_I = (s_{i_1}, \ldots, s_{i_m})$, and
\item Let $f_I : \overline{X} \rightarrow C_I$ be the product map
$f_I = (f_{i_1}, \ldots f_{i_m})$.
\end{enumerate}

Suppose that for each ordered set $I$, the point $s_I$ is a regular value for
$f_I$. Put $D_I = f_I^{-1}(s_I)$ and
$D = \bigcup_{i=1}^k D_{i}$. Then $D$ is an effective divisor with
simple normal crossings. Put $D_I^{\Int} = D_I - \bigcup_{I^\prime : 
|I^\prime| > |I|} D_{I^\prime}$. Then $D_I^{\Int}$ is a smooth
complex manifold of dimension $n - |I|$, possibly noncompact. 
Put $X = \overline X - D$.

Let $L_i$ be the holomorphic line bundle on $C_i$ associated to
$s_i$. Then $L_D = \bigotimes_{i=1}^k f_i^* L_i$ is the
holomorphic line bundle on $\overline{X}$ associated to $D$.
There is a holomorphic
section $\sigma_i$ of $L_i$ with zero set $s_i$, which is nondegenerate
at $s_i$. It is unique up to multiplication by a nonzero complex number.

\begin{remark}
In what follows, we could replace the point $s_i$ by a finite
subset of $C_i$, without any essential change.
\end{remark}

\begin{remark}
In \cite[Section 5]{tian-yau2}, Tian and Yau considered the case of
a smooth divisor, i.e. when $k=1$. They proved that if $K_{\overline{X}} + D
= 0$ then there is a complete Ricci-flat K\"ahler metric on $X$.
\end{remark}

\subsection{Cylindrical standard spatial asymptotics} \label{subsection3.1}

Suppose that $\overline{x} \in D_I^{\Int}$.
After permutation of indices, we can assume that
$\overline{x} \in (D_1 \cap D_2 \cap \ldots \cap D_m) -
(D_{m+1} \cup D_{m+2} \cup \ldots \cup D_k)$. We write
$0$ for $(0, \dots, 0) \in \Delta^n$. 
Let $p_i : \Delta^n \rightarrow \Delta$ be  projection onto the
$i$-th factor.
Then there are a neighborhood
$U$ of $\overline{x}$ in $\overline{X}$ and
a biholomorphic map
$F_{\overline{x}} : \Delta^n \rightarrow U$
so that
\begin{enumerate} [\upshape (i)]
\item For $i>m$, $U \cap D_i = \emptyset$,
\item $F_{\overline{x}}(0) = \overline{x}$, and
\item For $1 \le i \le m$, we have $f_i \circ F_{\overline{x}} = 
g_{s_i} \circ p_i$.
\end{enumerate}

In particular, $F_{\overline{x}} \left( (\Delta^*)^m \times
\Delta^{n-m} \right) = U \cap X$. The map $G_{\overline{x}}$ on
$\Delta^{n-m}$, given by $G_{\overline{x}}(w) = 
F_{\overline{x}}(0,w)$, is a biholomorphic map from $\Delta^{n-m}$
to a neighborhood of $\overline{x}$ in $D_I^{\Int}$.
Let $z^i$ be a local coordinate on $C_i$ around
$s_i$, which is the local inverse of
$g_{s_i}$. We also write $z^i$ for its pullback under
$f_i$ to a function on a neighborhood of $\overline{x}$.

Given $r \in (\R^+)^m$, let $\alpha_{r} : (\Delta^*)^m \rightarrow
(\Delta^*)^m$ be multiplication by $r$. If $Z$ is an auxiliary space
then we also write $\alpha_{r}$ for $(\alpha_{r}, \Id) : 
(\Delta^*)^m \times Z \rightarrow (\Delta^*)^m \times Z$.

\begin{defin} \label{3.1}
Let $\left\{ \omega_{D_I^{\Int}} \right\}$ be complete K\"ahler metrics
on $\left\{ D_I^{\Int} \right\}$. Let $\{c_i\}_{i=1}^k$ be positive
numbers.  Then $\omega_X$ has {\em cylindrical 
standard spatial asymptotics} associated
to $\left\{ \omega_{D_I^{\Int}} \right\}$ and $\{c_i\}_{i=1}^k$ if for
every $\overline{x} \in D_I^{\Int}$ and every local parametrization
$F_{\overline{x}}$,
\begin{equation} \label{3.2}
\lim_{r \rightarrow 0} \alpha_{r}^* F_{\overline{x}}^* \omega_X =
\sum_{i \in I} 2 c_i \sqrt{-1}
\frac{dz^i \wedge d\overline{z}^i}{|z_i|^2}
+ G_{\overline{x}}^* \omega_{D_I^{\Int}}.
\end{equation}
The limit in (\ref{3.2}) is taken in the
pointed $C^\infty$-topology around the basepoint 
$\left( \frac12, \ldots, \frac12 \right) \times 0 \in 
(\Delta^*)^m \times \Delta^{n-m}$. 
\end{defin}

\begin{proposition} \label{3.3} 
The notion of  
cylindrical standard spatial asymptotics in Definition \ref{3.1} 
is consistent under change of local coordinate.
\end{proposition}
\begin{proof}
Let $\{w^i\}_{i=1}^{n-m}$ be local coordinates for $D_{(1,\ldots,m)}$ around
$\overline{x}$.
If $\{\widehat{z}^i, \widehat{w}^i\}$ is a
different choice of coordinates then we can write
$\widehat{z}^i = \widehat{z}^i(z^i)$ and $\widehat{w}^i =
\widehat{w}^i(z,w)$. Let $\alpha_r$ be the operation of 
multiplying $z$ by $r$ and
let $\widehat{\alpha}_{r}$ be the operation of multiplying $\widehat{z}$ by
$r$. Then
\begin{equation} \label{3.4}
\alpha_r^* F_{\overline{x}}^* \omega_X = 
\alpha_r^* \left( \widehat{F}_{\overline{x}}^{-1} \circ
F_{\overline{x}} \right)^* (\widehat{\alpha}_r^*)^{-1} 
\widehat{\alpha}_r^* \widehat{F}_{\overline{x}}^* \omega_X.
\end{equation}
Suppose that $\omega_X$ has  cylindrical
standard spatial asymptotics with respect
to $(\widehat{z}, \widehat{w})$. Then
\begin{equation} \label{3.5}
\lim_{r \rightarrow 0} \widehat{\alpha}_{r}^* 
\widehat{F}_{\overline{x}}^* \omega_X =
\sum_{i=1}^m 2 c_i \sqrt{-1}
\frac{d\widehat{z}^i \wedge d\overline{\widehat{z}}^i}{|
\widehat{z}_i|^2}
+ \widehat{G}_{\overline{x}}^* \omega_{D_I^{\Int}}.
\end{equation}

Now 
\begin{equation} \label{3.6}
\left( \widehat{\alpha}_r^* \right)^{-1} \widehat{z}^i =
r^{-1} \widehat{z}^i.
\end{equation}
Expanding 
\begin{equation} \label{3.7}
\widehat{z}^i = a_{1,i} z^i + a_{2,i} \left( z^i \right)^2 + 
a_{3,i} \left( z^i \right)^3 + \ldots,
\end{equation}
with $a_{1,i} \neq 0$, we have
\begin{equation} \label{3.8}
\left( \widehat{F}_{\overline{x}}^{-1} \circ
F_{\overline{x}} \right)^* (\widehat{\alpha}_r^*)^{-1} \widehat{z}^i =
r^{-1} \left( a_{1,i} z^i + a_{2,i} \left( z^i \right)^2 + 
a_{3,i} \left( z^i \right)^3 + \ldots \right)
\end{equation}
Then
\begin{equation} \label{3.9}
\alpha_r^* \left( \widehat{F}_{\overline{x}}^{-1} \circ
F_{\overline{x}} \right)^* (\widehat{\alpha}_r^*)^{-1} \widehat{z}^i =
a_{1,i} z^i + a_{2,i} r \left( z^i \right)^2 + 
a_{3,i} r^2 \left( z^i \right)^3 + \ldots
\end{equation}
It follows that 
\begin{equation} \label{3.10}
\lim_{r \rightarrow 0} 
\alpha_r^* \left( \widehat{F}_{\overline{x}}^{-1} \circ
F_{\overline{x}} \right)^* (\widehat{\alpha}_r^*)^{-1} \widehat{z}^i =
a_{1,i} z^i. 
\end{equation}
Hence
\begin{equation} \label{3.11}
\lim_{r \rightarrow 0} 
\alpha_r^* \left( \widehat{F}_{\overline{x}}^{-1} \circ
F_{\overline{x}} \right)^* (\widehat{\alpha}_r^*)^{-1} 
\frac{d \widehat{z}^i \wedge d\overline{\widehat{z}^i}}{|\widehat{z}^i|^2} =
\frac{d{z}^i \wedge d\overline{z^i}}{|z^i|^2},
\end{equation}
with smooth pointed convergence around $\frac12 \in \Delta^*$.

Similarly, writing $\widehat{w} = \widehat{w}(z,w)$, we have
\begin{equation} \label{3.12}
\lim_{r \rightarrow 0} 
\alpha_r^* \left( \widehat{F}_{\overline{x}}^{-1} \circ
F_{\overline{x}} \right)^* (\widehat{\alpha}_r^*)^{-1} \widehat{w} =
\widehat{w}(0,w)
\end{equation}
Then
\begin{equation} \label{3.13}
\lim_{r \rightarrow 0} 
\alpha_r^* \left( \widehat{F}_{\overline{x}}^{-1} \circ
F_{\overline{x}} \right)^* (\widehat{\alpha}_r^*)^{-1} 
\widehat{G}_{\overline{x}}^* \omega_{D_I^{\Int}} 
 = {G}_{\overline{x}}^* \omega_{D_I^{\Int}}.
\end{equation}
In view of (\ref{3.4}) and (\ref{3.5}), the proposition follows.
\end{proof}

\begin{remark}
In the proof of Proposition \ref{3.3}, if we allowed more general
coordinate changes, of the form $\widehat{z}^i = \widehat{z}^i(z,w)$, then
the result of the proposition would definitely fail.  
This explains why we assume
that $\overline{X}$ fibers over curves, so that it makes sense to
consider coordinate changes of the form $\widehat{z}^i = \widehat{z}^i(z)$.
At the least, we need to assume that $D_i$ has a trivial normal bundle.
\end{remark}

\begin{proposition} \label{3.14}
If $\overline{X}$ admits a K\"ahler metric then $X$ admits
a complete K\"ahler metric with  cylindrical standard spatial
asymptotics.
\end{proposition}
\begin{proof}
This will follow from Proposition \ref{3.21}.
\end{proof}

We now prove Theorem \ref{intro2}, showing 
that the property of having cylindrical standard spatial
asymptotics is preserved under the K\"ahler-Ricci flow. 

\begin{proposition} \label{3.15}
Suppose that $\omega_X(0)$ has  cylindrical
standard spatial asymptotics associated
to $\{\omega_{D_I^{\Int}}(0) \}$ and $\{c_i\}_{i=1}^k$. Suppose that the
K\"ahler-Ricci flow $\omega_X(t)$, with initial K\"ahler form $\omega_X(0)$, 
exists on a maximal time interval $[0, T)$ in the sense of Theorem 
\ref{2.5}. Then for all $t \in [0,T)$, the metric $\omega_X(t)$ has 
 cylindrical standard spatial asymptotics associated to 
$\{\omega_{D_I^{\Int}}(t) \}$ and $\{c_i\}_{i=1}^k$, where
$\omega_{D_I^{\Int}}(t)$ is the K\"ahler-Ricci flow on
$D_I^{\Int}$ with initial K\"ahler form $\omega_{D_I^{\Int}}(0)$.
\end{proposition}
\begin{proof}
Take $\overline{x} \in D_I^{\Int}$.
After a change of labels we can assume that
$I = (1, \ldots, m)$. 
Let $F_{\overline{x}} : \left( \Delta^* \right)^m \times \Delta^{n-m}
\rightarrow \overline{X}$ be a local parametrization.
Let $r_j \rightarrow 0$ be any sequence.

From our assumptions, there is a uniform positive 
lower bound on the injectivity
radius of $F_{\overline{x}}^* \omega_X(0)$ at
$\alpha_{r_j} \left( \frac12, \ldots, \frac12, 0 \right)$
or, equivalently, on 
$\alpha_{r_j}^* F_{\overline{x}}^* \omega_X(0)$ at
$\left( \frac12, \ldots, \frac12, 0 \right)$.
Using the curvature bounds,
we can apply Hamilton's compactness
theorem \cite{Hamilton2} to extract a subsequence of 
$\left\{ \alpha_{r_j}^* F_{\overline{x}}^* \omega_X(\cdot) \right\}$
that converges to a Ricci flow solution
$\left( X_\infty, x_\infty, \omega_{X_\infty}(\cdot) \right)$,
defined on the time interval $[0, T)$.  There is a technical
issue that {\it a priori}, the solution is only smooth on
$(0, T)$, where we can apply Shi's local derivative estimates.
In order to get smoothness on $[0, T)$, we need
uniform bounds on the derivatives of the curvature tensor
of $\alpha_{r_j}^* F_{\overline{x}}^* \omega_X(t)$ for
$t$ in some interval $[0, \delta]$. 
Since we have uniform bounds on the derivatives of the curvature tensor
of $\alpha_{r_j}^* F_{\overline{x}}^* \omega_X(0)$ (on
time-$0$ metric balls around $\left( \frac12, \ldots, \frac12, 0 \right)$),
the local derivative estimate of
\cite[Appendix D]{Kleiner-Lott} gives the needed uniform bounds on the
$k$-th derivatives of the curvature tensor for small but positive time.
Then from the proof of \cite{Hamilton2}, we can say that after passing 
to a subsequence, 
there is a smooth pointed limit
\begin{align} \label{3.16}
& \lim_{j \rightarrow \infty} 
\left( \alpha_{r_j^{-1}} \left( (\Delta^*)^m \right) \times \Delta^{n-m}, 
\left( \frac12, \ldots, \frac12, 0 \right), 
\alpha_{r_j}^* F_{\overline{x}}^* \omega_X(\cdot) \right) = \\
& \left( (\C^*)^m \times \Delta^{n-m},
\left( \frac12, \ldots, \frac12, 0 \right), 
\omega_{\infty, \overline{x}}(\cdot) \right)
\notag
\end{align} 
for some K\"ahler-Ricci flow solution $\omega_{\infty, \overline{x}}(\cdot)$ 
that
exists on $(\C^*)^m \times \Delta^{n-m}$ for the time interval
$[0, T)$.

Covering $D_I^{\Int}$ by a locally finite collection 
$\left\{ G_{\overline{x}_k}
\left( \Delta^{n-m} \right) \right\}$ of charts, we can assume that
the Ricci flow solutions
$\left( (\C^*)^m \times \Delta^{n-m},
\left( \frac12, \ldots, \frac12, 0 \right), \omega_{\infty,
\overline{x}_k}(\cdot) \right)$
glue together to give a Ricci flow solution
$\left( (\C^*)^m \times D_I^{\Int}, \omega_{X_\infty}(\cdot)
\right)$ 
with complete time slices and bounded curvature on
compact time intervals; c.f.
\cite[Proof of Theorem 7.1]{lz-duke}.  
From the assumption of  cylindrical standard spatial
asymptotics,
\begin{equation} \label{3.17}
\omega_{X_\infty}(0) = \sum_{i=1}^m 2 c_i \sqrt{-1}
\frac{dz^i \wedge d\overline{z}^i}{|z_i|^2}
+ \omega_{D_I^{\Int}}(0).
\end{equation}
From the uniqueness of complete Ricci flow solutions with bounded curvature
on compact time intervals \cite{chen-zhu}, it follows that
\begin{equation} \label{3.18}
\omega_{X_\infty}(t) = \sum_{i=1}^m 2 c_i \sqrt{-1}
\frac{dz^i \wedge d\overline{z}^i}{|z_i|^2}
+ \omega_{D_I^{\Int}}(t).
\end{equation}

To return to the proof of the proposition,
suppose that its conclusion is not true.
Then for some $\overline{x} \in D_I^{\Int}$ and some $t \in [0, T)$,
there is a sequence $r_j \rightarrow 0$ with the property that
even after passing to any subsequence,
$\alpha_{r_j}^* F_{\overline{x}}^* \omega_X(t)$ does not converge to
$\sum_{i=1}^m 2 c_i \sqrt{-1}
\frac{dz^i \wedge d\overline{z}^i}{|z_i|^2}
+ G_{\overline{x}}^* \omega_{D_I^{\Int}}(t)$
in the pointed $C^\infty$-topology. Here the basepoint is
$\left( \frac12, \ldots, \frac12, 0 \right) \in 
\left( \Delta^* \right)^m \times \Delta^{n-m}$.
However, taking $\overline{x}_1 = \overline{x}$ in the above
construction, we have shown that after passing to a subsequence, 
$\lim_{j \rightarrow \infty} \alpha_{r_j}^* F_{\overline{x}}^* \omega_X(t)
=\sum_{i=1}^m 2 c_i \sqrt{-1}
\frac{dz^i \wedge d\overline{z}^i}{|z_i|^2}
+ G_{\overline{x}}^* \omega_{D_I^{\Int}}(t)$
in the pointed $C^\infty$-topology.
This is a contradiction, thereby proving the proposition.
\end{proof}

\subsection{Cylindrical superstandard spatial asymptotics}
\label{subsection3.2}

Let $h_{i}$ be a Hermitian metric on $L_i$.

\begin{defin} \label{3.19}
A K\"ahler metric $\omega_X$ on $X$ has  {\em cylindrical superstandard spatial
asymptotics} associated to $\{h_i\}_{i=1}^k$
if it has  cylindrical
standard spatial asymptotics (associated to $\{\omega_{D_I^{int}} \}$
and $\{ c_i \}_{i=1}^k$) and
\begin{equation} \label{3.20}
\omega_X = \eta_{\overline{X}} +
\sqrt{-1} \partial \overline{\partial} \left( \sum_{i=1}^k c_i 
f_i^* \log^2 \left| \sigma_i \right|_{h_i}^{-2} +
H \right), 
\end{equation}
where 
\begin{enumerate} [\upshape (i)]
\item $\eta_{\overline{X}}$ is a smooth closed $(1,1)$-form on
$\overline{X}$, and
\item $H \in C^\infty(X) \cap L^\infty(X)$. 
\end{enumerate}
\end{defin}

Note that in Definition \ref{3.19}, 
the choice of $\{h_i\}_{i=1}^k$ does matter.

\begin{proposition} \label{3.21}
If $\overline{X}$ admits a K\"ahler metric then $X$ admits
a complete K\"ahler metric with  cylindrical superstandard spatial
asymptotics.
\end{proposition}
\begin{proof}
We have
\begin{equation}
\sqrt{-1} \partial \overline{\partial}  
\log^2 |\sigma_i|_{h_i}^{-2} = 
2 \sqrt{-1} \partial \log |\sigma_i|_{h_i}^{-2} \wedge
\overline{\partial} |\sigma_i|_{h_i}^{-2} + 2 
\left( \log |\sigma_i|_{h_i}^{-2} \right) F_{h_i},
\end{equation}
where $F_{h_i}$ is the curvature $2$-form associated to $h_i$.
In terms of a local coordinate $z^i$ around $s^i$,
the right-hand side is asymptotic to $2 \sqrt{-1}
\frac{dz^i \wedge d\overline{z}^i}{|z^i|^2}$ as $z^i \rightarrow 0$.

Let $\omega_{\overline{X}}$ be a K\"ahler metric on $\overline{X}$.
Given a parameter $K < \infty$ and positive
constants $\{c_i\}_{i=1}^k$, put
\begin{equation} \label{3.22}
\omega_X = 
\sqrt{-1} \partial \overline{\partial} \left(
\sum_{i=1}^k c_i f_i^* \log^2 |\sigma_i|_{h_i}^{-2} \right)
+ K \omega_{\overline{X}}.
\end{equation}
Taking $K$ sufficiently large, $\omega_X$ is a complete K\"ahler metric on $X$
with cylindrical standard spatial asymptotics.
Then it clearly also has cylindrical superstandard spatial asymptotics.
\end{proof}

We now show that the property of having cylindrical superstandard spatial
asymptotics is preserved under the K\"ahler-Ricci flow. 

\begin{proposition} \label{3.23}
Suppose that $\omega_X(0)$ has 
 cylindrical superstandard spatial asymptotics associated
to $\{h_i\}_{i=1}^k$.
Suppose that the K\"ahler-Ricci flow $\omega_X(t)$,
with initial K\"ahler metric $\omega_X(0)$, exists on a maximal time
interval $[0,T)$ in the sense of Theorem \ref{2.5}. Then for all
$t \in [0,T)$, $\omega_X(t)$ has  cylindrical
superstandard spatial asymptotics,
associated to $\{h_i\}_{i=1}^k$.
\end{proposition}
\begin{proof}
Choose a Hermitian metric $h_{K_{\overline{X}} \otimes L_D}$
on $K_{\overline{X}} \otimes L_D$.
Along with $\{h_i\}_{i=1}^k$, we obtain a Hermitian metric 
$h_{K_{\overline{X}}}$ on $K_{\overline{X}}$.
Then
\begin{equation} \label{3.24}
\Ric (\omega_X(0))  =  - \sqrt{-1} F \left( 
h_{K_{\overline{X}} \otimes L_D} \right)
 - \sqrt{-1} \partial \overline{\partial} \left(
\log
\frac{
h_{K_{\overline{X}}} \prod_{i=1}^k f_i^* |\sigma_i|_{h_i}^2
}{
h_{K_X}
} \right) 
\end{equation}
on $X$. 

Put $\eta^\prime_{\overline{X}} = - \sqrt{-1} F \left( 
h_{K_{\overline{X}} \otimes L_D} \right)$ and $H^\prime =
\log \frac{
h_{K_{\overline{X}}} \prod_{i=1}^k f_i^* |\sigma_i|_{h_i}^2
}{
h_{K_X}
}$. By the  cylindrical
standard spatial asymptotics, $H^\prime \in C^\infty(X) \cap
L^\infty(X)$.

Recall the definition of $\omega_t$ from (\ref{2.2}). We can write
\begin{align} \label{3.25}
\omega_X(t) \: = \: & \omega_t + \sqrt{-1} \partial \overline{\partial} u(t) \\
\: = \: & \eta_{\overline{X}} - t \eta^\prime_{\overline{X}} + \sqrt{-1}
\partial \overline{\partial} \left( \sum_{i=1}^k c_i
f_i^* \log^2 \left| \sigma_i \right|_{h_i}^{-2} +
H + t H^\prime + u(t) \right). \notag
\end{align}
Since $u(t) \in C^\infty(X) \cap L^\infty(X)$, the proposition follows.
\end{proof}

If $\omega_X(0)$ has cylindrical superstandard spatial asymptotics then
we would like to give a characterization of the first singularity
time, if there is one, by making Theorem \ref{2.5} more explicit.
For the ``cuspidal'' asymptotics considered in
\cite{lz-duke}, the manifold $(X, \omega_X(0))$ had finite volume.
Transplanting $\omega_X(0)$ to $\overline{X}$, we obtained a
closed $(1,1)$-current, which represented a cohomology class on
$\overline{X}$. In the present case, $(X, \omega_X(0))$ has
infinite volume and we cannot directly obtain a cohomology class
on $\overline{X}$. However, we can subtract the leading singularity,
which is $\sqrt{-1} \partial \overline{\partial}$ of a function,
and thereby define a renormalized cohomology class in $\overline{X}$.

With reference to Definition \ref{3.19}, since $H$ is a smooth bounded
function on $X$, it extends by zero to an integrable function on
$\overline{X}$. Then $\sqrt{-1} \partial \overline{\partial} H$ is
a closed $(1,1)$-current on $\overline{X}$ (which is cohomologically
trivial). Hence the form on $X$ given by 
$\omega_X(0) -
\sqrt{-1} \partial \overline{\partial} \left( \sum_{i=1}^k c_i 
f_i^* \log^2 \left| \sigma_i \right|_{h_i}^{-2} \right)$, which
equals $\eta_{\overline{X}} + \sqrt{-1} \partial \overline{\partial} H$,
has a natural extension to a closed $(1,1)$-current on $\overline{X}$.

The relevant ring of functions, for cylindrical asymptotics, can
be characterized in the following way.

\begin{defin} \label{3.26}
The ring $C^\infty_{\cyl}(X)$ consists of the smooth functions $f$ on
$X = \overline{X} - D$ so that for every $\overline{x} \in D$ and
every local parametrization $F_{\overline{x}}$, the pullback
$F_{\overline{x}}^* f \in C^\infty((\Delta^*)^m \times \Delta^{n-m})$
has the property that for any multi-index $(l_1, \overline{l}_1, \ldots,
l_{n}, \overline{l}_{n})$, the function
\begin{align} \label{3.27}
& \left(z^1  \frac{\partial}{\partial z^1} \right)^{l_1}
\left( \overline{z}^1 \frac{\partial}{\partial \overline{z}^1} 
\right)^{\overline{l}_1}
\ldots
\left(z^m  \frac{\partial}{\partial z^m} \right)^{l_m}
\left( \overline{z}^m \frac{\partial}{\partial \overline{z}^m} 
\right)^{\overline{l}_m} \\
& 
\left( \frac{\partial}{\partial w^1} \right)^{l_{m+1}}
\left( \frac{\partial}{\partial \overline{w}^1} 
\right)^{\overline{l}_{m+1}}
\ldots
\left( \frac{\partial}{\partial w^{n-m}} \right)^{l_n}
\left( \frac{\partial}{\partial \overline{w}^{n-m}} 
\right)^{\overline{l}_n}
F_{\overline{x}}^* f \notag
\end{align}
is uniformly bounded.
\end{defin}

\begin{proposition} \label{3.28}
Suppose that $\omega_X(0)$ has  cylindrical
superstandard spatial asymptotics
associated to $\{h_i\}_{i=1}^k$.
Let $\eta_{\overline{X}} \in \Omega^{(1,1)}(\overline{X})$ 
be a smooth representative of the cohomology class represented by the
closed current
\begin{equation} \label{3.29}
\omega_X(0) -
\sqrt{-1} \partial \overline{\partial} \left( \sum_{i=1}^k c_i 
f_i^* \log^2 \left| \sigma_i \right|_{h_i}^{-2} \right)
\end{equation}
on $\overline{X}$.
Let $\eta^\prime_{\overline{X}} \in \Omega^{(1,1)}(\overline{X})$ be
a smooth representative of 
$- 2 \pi [K_{\overline{X}} + D] \in \HH^{(1,1)}(\overline{X})$.
Let $T_3$ be the supremum (possibly infinite) of the numbers $T^\prime$
for which there is some
$f_{T^\prime} \in C^\infty_{\cyl}(X)$ so that
\begin{equation} \label{3.30}
\eta_{\overline{X}} - T^\prime \eta^\prime_{\overline{X}} + 
\sqrt{-1} \partial \overline{\partial} \left( \sum_{i=1}^k c_i 
f_i^* \log^2 \left| \sigma_i \right|_{h_i}^{-2} + f_{T^\prime} \right)
\end{equation}
is a K\"ahler form on $X$ which is biLipschitz to $\omega_X(0)$.
Then $T_3$ equals the numbers $T_1 = T_2$ of Theorem \ref{2.5}.
\end{proposition}
\begin{proof}
Let $\omega_X(0)$ and $T^\prime$ be as in the statement of the
proposition.  Since the present $\eta_{\overline{X}}$ and the
$\eta_{\overline{X}}$ of Definition \ref{3.19} differ by 
$\sqrt{-1} \partial \overline{\partial}$ of a smooth function on
$\overline{X}$, we can still write
\begin{equation} \label{3.31}
\omega_X(0) = \eta_{\overline{X}} +
\sqrt{-1} \partial \overline{\partial} \left( \sum_{i=1}^k c_i 
f_i^* \log^2 \left| \sigma_i \right|_{h_i}^{-2} +
H \right)
\end{equation}
for some $H \in C^\infty(X) \cap L^\infty(X)$. With reference to the
proof of Proposition \ref{3.23}, as the present
$\eta^\prime_{\overline{X}}$ differs from 
$- \sqrt{-1} F \left( 
h_{K_{\overline{X}} \otimes L_D} \right)$ by
$\sqrt{-1} \partial \overline{\partial}$ of a smooth function on
$\overline{X}$, we can still write 
\begin{equation} \label{3.32}
\omega_{T^\prime}
= \eta_{\overline{X}} - T^\prime \eta^\prime_{\overline{X}} + \sqrt{-1}
\partial \overline{\partial} \left( \sum_{i=1}^k c_i
f_i^* \log^2 \left| \sigma_i \right|_{h_i}^{-2} +
H + T^\prime H^\prime \right)
\end{equation}
for some $H^\prime \in C^\infty(X) \cap L^\infty(X)$.
Then
\begin{equation} \label{3.33}
 \eta_{\overline{X}} - T^\prime \eta^\prime_{\overline{X}} + \sqrt{-1}
\partial \overline{\partial} \left( \sum_{i=1}^k c_i
f_i^* \log^2 \left| \sigma_i \right|_{h_i}^{-2} + f_{T^\prime}  \right) =
\omega_{T^\prime}
+  \sqrt{-1}
\partial \overline{\partial} \left(
 f_{T^\prime} - H - T^\prime H^\prime \right).
\end{equation}
Put $F_{T^\prime} =  f_{T^\prime} -
H - T^\prime H^\prime$. By assumption, the left-hand side of
(\ref{3.33}), and hence also 
$\omega_{T^\prime}
+  \sqrt{-1}
\partial \overline{\partial} 
F_{T^\prime}$, is a K\"ahler metric
which is biLipschitz to $\omega_X(0)$.
Since $\omega_X(0)$ has  cylindrical
standard spatial asymptotics, it follows that
\begin{equation} \label{3.34}
-  \const \omega_X(0) \le 
\sqrt{-1}
\partial \overline{\partial} 
F_{T^\prime} \le \const \omega_X(0).
\end{equation}
Hence 
\begin{equation} \label{3.35}
| \triangle_{\omega_X(0)} F_{T^\prime} | \le \const
\end{equation}
Since $\omega_X(0)$ has bounded geometry (including a positive injectivity
radius), elliptic regularity implies that for each $k$, the
$k$th covariant derivatives of $F_{T^\prime}$ (with respect to 
$\omega_X(0)$) are uniformly bounded. 
It follows that $T_3 \le T_2$.

Now suppose that $T^\prime$ is as in the definition of $T_1$ in Theorem
\ref{2.5}.
As $\omega_X(0)$ has  cylindrical
superstandard spatial asymptotics, we can write
\begin{equation} \label{3.36}
\omega_{T^\prime} + \sqrt{-1}
\partial \overline{\partial} u(T^\prime)
= \eta_{\overline{X}} - T^\prime \eta^\prime_{\overline{X}} + \sqrt{-1}
\partial \overline{\partial} \left( \sum_{i=1}^k c_i
f_i^* \log^2 \left| \sigma_i \right|_{h_i}^{-2} +
H + T^\prime H^\prime + u(T^\prime) \right)
\end{equation}
for some $H, H^\prime \in C^\infty(X) \cap L^\infty(X)$.
Put $f_{T^\prime} = H + T^\prime H^\prime + u(T^\prime)$, so
\begin{equation} \label{3.37}
\eta_{\overline{X}} - T^\prime \eta^\prime_{\overline{X}} + \sqrt{-1}
\partial \overline{\partial} \left( \sum_{i=1}^k c_i
f_i^* \log^2 \left| \sigma_i \right|_{h_i}^{-2} +
f_{T^\prime} \right) =
\omega_{T^\prime} + \sqrt{-1}
\partial \overline{\partial} u(T^\prime).
\end{equation}
From Proposition \ref{3.23},
$\omega_{T^\prime} +  \sqrt{-1} \partial \overline{\partial} 
u(T^\prime)$ has  cylindrical
superstandard spatial asymptotics.
As before, using elliptic regularity we conclude that for each 
$k$, the $k$th covariant derivatives of 
$f_{T^\prime}$ (with respect to $\omega_X(0)$) are uniformly 
bounded.  This is equivalent to saying that $f_{T^\prime} \in 
C^\infty_{\cyl}(X)$.

Thus $T_1 \le T_3$. This proves the proposition.
\end{proof}

\begin{corollary} \label{3.38}
Suppose that $\omega_X(0)$ has cylindrical superstandard spatial
asymptotics.  If $[K_{\overline{X}} + D] \ge 0$ then the flow exists
for all positive time.
\end{corollary}
\begin{proof}
With reference to Proposition \ref{3.28},
we can choose $- \eta_{\overline{X}}^\prime$ to be a nonnegative closed
$(1,1)$-form. Since $T_3 > 0$, there is some $f_0 \in C^\infty_{\cyl}(X)$
so that 
$\eta_{\overline{X}} + 
\sqrt{-1} \partial \overline{\partial} \left( \sum_{i=1}^k c_i 
f_i^* \log^2 \left| \sigma_i \right|_{h_i}^{-2} + f_{0} \right)$
is a K\"ahler form on $X$ which is biLipschitz to $\omega_X(0)$.
Then for any $T^\prime > 0$, 
$\eta_{\overline{X}} - T^\prime \eta^\prime_{\overline{X}} + 
\sqrt{-1} \partial \overline{\partial} \left( \sum_{i=1}^k c_i 
f_i^* \log^2 \left| \sigma_i \right|_{h_i}^{-2} + f_{0} \right)$
is also a K\"ahler form on $X$ which is biLipschitz to $\omega_X(0)$.
Using Proposition \ref{3.28}, this proves the corollary.
\end{proof}

\begin{remark}
Proposition \ref{3.28} is only partly a statement about
$\overline{X}$, since the definition of $T_3$ is a statement
about K\"ahler metrics on $X$ with certain properties.  In this
sense, the proposition is not as definitive as the corresponding
statement about cuspidal asymptotics in \cite[Theorem 8.19]{lz-duke}.
\end{remark}

\section{Bulging K\"ahler-Ricci flows} \label{section4}

This section deals with bulging spatial asymptotics. In
Subsection \ref{subsection4.1} we define the notion of bulging standard
spatial asymptotics, and show that it is preserved under the
K\"ahler-Ricci flow, with no change in the divisor metric.
In Subsection \ref{subsection4.2} we consider parabolic rescalings around a
sequence of points that go to spatial infinity in the 
time-zero slice. Under a decay assumption on the Ricci curvature,
we show that the limit is a product flow which exhibits the
K\"ahler-Ricci flow on the divisor.

In Subsection \ref{subsection4.3} we assume that $D$ is ample.
We introduce bulging superstandard
spatial asymptotics. We show that if $\overline{X}$ admits a
K\"ahler metric then $X$ admits a metric with
bulging superstandard spatial asymptotics.
We prove that having bulging superstandard spatial
asymptotics is preserved under the K\"ahler-Ricci flow. 
We characterize the first singularity time, if there is one.

Let $\overline{X}$ be a compact connected $n$-dimensional complex manifold.
Let $D$ be a smooth effective divisor in $\overline{X}$.
Let $L_D$ be the holomorphic line bundle on $\overline{X}$ associated
to $D$. There is a holomorphic section $\sigma$ of $L_D$ with zero set
$D$, which is nondegenerate at $D$.  It is unique up to multiplication
by a nonzero complex number.

\subsection{Bulging standard spatial asymptotics} \label{subsection4.1}

Given $\overline{x} \in D$,
there are a neighborhood
$U$ of $\overline{x}$ in $\overline{X}$ and
a biholomorphic map
$F_{\overline{x}} : \Delta^n \rightarrow U$
so that
\begin{enumerate} [\upshape (i)]
\item $F_{\overline{x}}(0) = \overline{x}$, and
\item $F_{\overline{x}} \left( \Delta^* \times
\Delta^{n-1} \right) = U \cap X$.
\end{enumerate}

The map $G_{\overline{x}}$ on
$\Delta^{n-1}$, given by $G_{\overline{x}}(w) = 
F_{\overline{x}}(0,w)$, is a biholomorphic map from $\Delta^{n-1}$
to a neighborhood of $\overline{x}$ in $D$.
Let $z$ be the local coordinate on $U$ coresponding to the first factor
in $\Delta^n$.

Using the covering map $H \rightarrow \Delta^*$, given by
$u \rightarrow e^{\sqrt{-1} u}$, 
let $\widetilde{F}_{\overline{x}} : H \times \Delta^{n-1} \rightarrow U \cap X$
be the lift of $F_{\overline{x}} \Big|_{\Delta^* \times \Delta^{n-1}}$.

Given $r \in \R^+$, let $\alpha_{r} : \C \rightarrow \C$ be the
map $\alpha_r(u) = ru + \sqrt{-1} \frac{r^2}{2}$.
If $Z$ is an auxiliary space
then we also write $\alpha_{r}$ for $(\alpha_{r}, \Id) : 
\C \times Z \rightarrow \C \times Z$.

\begin{defin} \label{5.1}
Let $\omega_D$ be a K\"ahler metric
on $D$. Given $N > 0$,
$\omega_X$ has {\em bulging standard spatial asymptotics} associated
to $(\omega_D, N)$ if for
every $\overline{x} \in D$ and every local parametrization
$F_{\overline{x}}$,
\begin{equation} \label{5.2}
\lim_{r \rightarrow \infty} 
r^{- \frac{2}{N}}
\alpha_{r}^* 
\widetilde{F}_{\overline{x}}^* \omega_X =
\frac12 \left( \frac{N+1}{N} \right)^2
\left(
\sqrt{-1} du \wedge d\overline{u} +
N G_{\overline{x}}^* \omega_D
\right).
\end{equation}
The limit in (\ref{5.2}) means smooth convergence on
any subset $\{z \in \C : |z| < S\} \times \Delta^{n-1}$
of $\C \times \Delta^{n-1}$. 
\end{defin}

\begin{remark}
Note that although $\widetilde{F}_{\overline{x}}$ is originally defined
on $H \times \Delta^{n-1}$, the limit is taken
around the basepoint $(0,0)$ in $\C \times \Delta^{n-1}$. This makes
sense because $\alpha_r(0) = \sqrt{-1} \frac{r^2}{2} \in H$.
\end{remark}

\begin{proposition}
If $\omega_X$ has bulging standard spatial asymptotics then to leading order,
as $z \rightarrow 0$, the metric in local coordinates has the form
\begin{equation} \label{5.3}
\omega_X \sim \frac12 \left( \frac{N+1}{N} \right)^2  
\left( \log |z|^{-2} \right)^{\frac{1}{N}} 
\left( \sqrt{-1} \frac{dz \wedge 
d\overline{z}}{|z|^2 \log |z|^{-2}} + N \omega_D \right).
\end{equation}
\end{proposition}
\begin{proof}
Relating the local function $z$ on $U$ to the local function
$u$ on $\C$ by $\alpha_r^* \widetilde{F}_{\overline{x}}^*$, we have
\begin{equation} 
z = e^{\sqrt{-1} \left( ru + \sqrt{-1} \frac{r^2}{2} \right)},
\end{equation}
so 
\begin{equation}
\frac{dz}{z} = \sqrt{-1} \: r \: du
\end{equation}
and
\begin{equation}
\log |z|^{-2} = r^2 - \sqrt{-1} r (u - \overline{u}).
\end{equation}
Then
\begin{align}
& \frac12 \left( \frac{N+1}{N} \right)^2
r^{\frac{2}{N}}
\left(
\sqrt{-1} du \wedge d\overline{u} +
N G_{\overline{x}}^* \omega_D \right) = \\
& \frac12 \left( \frac{N+1}{N} \right)^2
r^{\frac{2}{N}}
\left(
\sqrt{-1} \frac{dz \wedge d\overline{z}}{r^2 |z|^2} +
N G_{\overline{x}}^* \omega_D \right) \sim \notag \\
& \frac12 \left( \frac{N+1}{N} \right)^2 (\log |z|^{-2})^{\frac{1}{N}}
\left(
\sqrt{-1} \frac{dz \wedge d\overline{z}}{|z|^2 \log |z|^{-2}} +
N \omega_D \right). \notag
\end{align}
This proves the proposition.
\end{proof}

One can check that 
the notion of bulging standard spatial asymptotics in
Definition \ref{5.1} is consistent under change of local coordinate.

Going out the end of $X$ corresponds to taking $z \rightarrow 0$.
The distance from a basepoint in $X$ is asymptotic to
\begin{equation} \label{5.4}
R \sim \left( \log |z|^{-2} \right)^{\frac{N+1}{2N}}.
\end{equation}
Fix $\overline{x} \in D$ and let $\theta$ be the angular
coordinate of $z$.
Then as $z \rightarrow 0$, i.e. as $R \rightarrow \infty$,
the metric on $X$ is asymptotic to
\begin{equation} \label{5.5}
g_X \sim dR^2 + \left( \frac{N+1}{N} \right)^2 
R^{-2 \frac{N-1}{N+1}} d\theta^2 + 
\frac{(N+1)^2}{2N} R^{\frac{2}{N+1}} g_D(\overline{x}).
\end{equation}
The sectional curvatures decay as
\begin{equation} \label{5.6}
|\Rm| = O \left( R^{- \frac{2}{N+1}} \right).
\end{equation}
For a given
$\overline{x} \in D$, as $z \rightarrow 0$,
the geometry comes closer and closer to having a product structure.

We now show that the property of having bulging standard spatial asymptotics
is preserved under the K\"ahler-Ricci flow, with no change in the
divisor metric.

\begin{proposition} \label{5.7}
Let $\omega_X(\cdot)$ be a K\"ahler-Ricci flow defined for $t \in [0, T)$,
with bounded curvature on compact time intervals.
Suppose that $\omega_X(0)$ has bulging
standard spatial asymptotics associated to
$(\omega_D, N)$. Then for all $t \in [0, T)$, the metric
$\omega_X(t)$ has bulging standard spatial asymptotics associated to
$(\omega_D, N)$. 
\end{proposition}
\begin{proof}
Suppose that the conclusion of the proposition is not true. Then there are some
$t \in [0, T)$, $\overline{x} \in D$, 
${F}_{\overline{x}} :
\Delta^* \times \Delta^{n-1} \rightarrow \overline{X}$, $S < \infty$ and 
$r_i \rightarrow \infty$ with the property that even 
after passing to any subsequence of the
$r_i$'s, 
the metrics 
$r_i^{- \frac{2}{N}}
\alpha_{r_i}^* 
\widetilde{F}_{\overline{x}}^* \omega_X(t)$ do not smoothly
converge to 
$\frac12 \left( \frac{N+1}{N} \right)^2
\left(
\sqrt{-1} du \wedge d\overline{u} +
N G_{\overline{x}}^* \omega_D
\right)$  on $\{z \in \C : |z| < S \} \times \Delta^{n-1}$.

From the bulging standard spatial asymptotics of $\omega_X(0)$, 
given $k \ge 0$, we have $|\nabla^k \Rm|(x,0) =
O \left( d_0(x,x_0)^{-\frac{k+2}{N+1}} \right)$. 
We now want to show that 
\begin{equation} \label{4.13.5}
|\nabla^k \Rm|(x,t) =
O \left( d_0(x,x_0)^{-\frac{k+2}{N+1}} \right).
\end{equation}
There are two arguments for this.  First, given a large
positive number $A$, put
\begin{equation}
\phi(x) = 
\left(A +  \log |\sigma(x)|^{-2} \right)^{\frac{N+1}{2N}}.
\end{equation}
The definition of $\phi$ is motivated by (\ref{5.4}).
Then $\phi$ is a distance-like function that
satisfies (\ref{4.3.5}).
Proposition \ref{4.7} implies that (\ref{4.13.5}) holds.

Alternatively, we can apply Proposition \ref{B.15} to see that 
(\ref{4.13.5}) holds. Either way, we obtain uniform bounds on the curvature of 
$r_i^{- \frac{2}{N}} \alpha_{r_i}^* 
\widetilde{F}_{\overline{x}}^* \omega_X(t)$ and its covariant derivatives,
when considered on any fixed $\{z \in \C : |z| \le S^\prime\} 
\times \Delta^{n-1}$,
as $i \rightarrow \infty$. 
After passing to a subsequence, we can assume that 
there is a Riemannian metric $\omega_\infty(t)$ on 
$\C \times \Delta^{n-1}$ so that the Riemannian
metrics
$\left\{ r_i^{- \frac{2}{N}} \alpha_{r_i}^* 
\widetilde{F}_{\overline{x}}^* \omega_X(t)
\right\}_{i=1}^\infty$
converge smoothly to $\omega_\infty(t)$ on each
$\{z \in \C : |z| < S^\prime\} \times \Delta^{n-1}$.

On the time interval $[0,t]$, the curvature of $X$ decays uniformly
as $O \left( d_0(x,x_0)^{- \frac{2}{N+1}} \right)$. It follows from the
Ricci flow equation that at a point $x \in X$,
the metrics $\omega_X(t)$ and $\omega_X(0)$ are
$e^{\const t  d_0(x,x_0)^{- \frac{2}{N+1}}}$-biLipschitz to each other.
Applying this to small neighborhoods of points
$x_i = \widetilde{F}_{\overline{x}}(\alpha_{r_i}(0,w)) =
F_{\overline{x}} \left( e^{- \frac{r_i^2}{2}}, w \right)$,
since $d_0(x_i, x_0) \rightarrow \infty$,
it follows that $\omega_\infty(t)$ is isometric to the
corresponding $\omega_\infty(0)$. The latter is the product metric 
$\frac12 \left( \frac{N+1}{N} \right)^2
\left( \sqrt{-1} du \wedge d\overline{u} + N 
G_{\overline{x}}^* \omega_D \right)$
on $\C \times \Delta^{n-1}$.
Hence as $i \rightarrow \infty$, the metrics
$r_i^{- \frac{2}{N}}
\alpha_{r_i}^* 
\widetilde{F}_{\overline{x}}^* \omega_X(t)$ smoothly converge to
$\frac12 \left( \frac{N+1}{N} \right)^2
\left(
\sqrt{-1} du \wedge d\overline{u} +
N G_{\overline{x}}^* \omega_D
\right)$ on $\{z \in \C : |z| < S\} \times \Delta^{n-1}$. 
This is a contradiction.
\end{proof}

\subsection{Parabolic rescaling of bulging asymptotics} \label{subsection4.2}

As mentioned in the introduction and seen in Proposition \ref{5.7},
if the initial metric has bulging standard spatial asymptotics then
the divisor flow does not enter into the asymptotics of the
K\"ahler-Ricci flow on $X$ for any finite time interval. In order
to see the divisor flow, we must rather do parabolic rescalings
around points $x_i$ that tend to spatial infinity in the initial time slice.
The sectional curvature at $x_i$ decays like
$O \left( d_0(x_i, x_0)^{- \frac{2}{N+1}} \right)$, so the relevant time
scale increases like $d_0(x_i, x_0)^{\frac{2}{N+1}}$.

In order to take a limit of the parabolic rescalings, we need uniform
curvature estimates on balls centered around $x_i$ of radius comparable
to $d_0(x_i, x_0)^{\frac{1}{N+1}}$, and on a time interval of length comparable
to $d_0(x_i, x_0)^{\frac{2}{N+1}}$. Such estimates would follow from
pseudolocality, if we could apply it.  Unfortunately, because of
the shrinking circle fiber, the injectivity radius on such
time-zero balls
goes to zero as $i \rightarrow \infty$ even before rescaling, and
even more so after rescaling.  The curvature estimates from pseudolocality 
can definitely fail in such a situation \cite[Section 4]{Lu}.

Because one cannot apply pseudolocality, it appears that one must
make some assumption to ensure the needed sectional curvature bounds.
There is some flexibility in the precise assumption made.  We
assume a decay of the Ricci curvature on the relevant spacetime
region and show that it implies the needed sectional curvature bound.
The Ricci curvature assumption implies a uniform biLipschitz bound
within the parabolic ball. Using Appendix \ref{localapp}, this gives the
sectional curvature bound.

\begin{proposition} \label{5.8}
Let $\omega_X(\cdot)$ be a K\"ahler-Ricci flow defined for $t \in [0,
\infty)$, with bounded curvature on compact time intervals.  Suppose that
$\omega_X(0)$ has bulging standard spatial asymptotics associated to
$(\omega_D(0), N)$. Let $\omega_D(\cdot)$ denote the Ricci flow on $D$.

Given $A < \infty$, 
suppose that 
\begin{equation} \label{5.9}
|\Ric(x,t)| = O \left( d_0(x,x_0)^{- \frac{2}{N+1}} \right)
\end{equation}
uniformly
on the spacetime region $\left\{ (x,t) \: : \: t \: \le \: A \: 
d_0(x,x_0)^{\frac{2}{N+1}} \right\}$. 
Given $\overline{x} \in D$,
${F}_{\overline{x}} :
\Delta^* \times \Delta^{n-1} \rightarrow \overline{X}$ and 
$r>0$, consider the Ricci flow
$\omega_r(t) = 
r^{- \frac{2}{N}}
\alpha_{r}^* 
\widetilde{F}_{\overline{x}}^* \omega_X(r^{\frac{2}{N}} t)$.
Then 
\begin{equation} \label{5.10}
\lim_{r \rightarrow \infty} \omega_r(\cdot) =  
\frac12 \left( \frac{N+1}{N} \right)^2
\left(
\sqrt{-1} du \wedge d\overline{u} +
N G_{\overline{x}}^* \omega_D(\cdot)
\right),
\end{equation}
with smooth convergence on the product of the
time interval $[0, A]$ with any  $\{z \in \C : |z| < S\} 
\times \Delta^{n-1}$.
\end{proposition}
\begin{proof}
Let $x_0$ be the basepoint in $X$. As before, the
Ricci flow equation and (\ref{5.6}) imply that
at any $x \in X$, the metrics $\omega_X(t)$ and $\omega_X(0)$ are
$e^{\const \: t \: d_0(x,x_0)^{- \frac{2}{N+1}}}$-biLipschitz.
Using the bulging standard
spatial asymptotics of $\omega_X(0)$ and Proposition \ref{A.1} of
the appendix (with the parameter $r$ of the proposition equal to
$d_0(x,x_0)^{\frac{1}{N+1}}$),
the biLipschitz bounds imply a uniform curvature
bound
\begin{equation}
|\Rm(x,t)| = O \left( d_0(x,x_0)^{- \frac{2}{N+1}} \right)
\end{equation}
on $\left\{ (x,t) \: : \: t \: \le \: A \: 
d_0(x,x_0)^{\frac{2}{n+1}} \right\}$.
Similarly, Proposition \ref{A.38} gives 
\begin{equation}
|\nabla^k \Rm|(x,t) = O \left( d_0(x,x_0)^{- \frac{2(k+1)}{N+1}} \right).
\end{equation}

Let $\overline{x} \in D$ and $F_{\overline{x}}$ be as in the
statement of the proposition. 
The curvature bounds imply that if $r_i \rightarrow \infty$ then we can take a
subsequence of $\{\omega_{r_i}(\cdot) \}_{i=1}^\infty$
that converges smoothly to a Ricci flow solution defined on
$\C \times \Delta^{n-1} \times [0,A]$.

This process can be globalized with respect to the divisor $D$; 
c.f. \cite[Proof of Theorem 7.1]{lz-duke}.
The result is that for an arbitrary $\overline{x} \in D$, there is
a subsequence of the pointed Ricci flows
$\left\{ \left( r_i^{- \frac{2}{N}} \omega_X( r_i^{\frac{2}{N}} \cdot), 
{F}_{\overline{x}} ( e^{-\frac{r_i^2}{2}}, 0 ) \right)
\right\}_{i=1}^\infty$
that converges, in
the sense of Ricci flows on \'etale groupoids, to a limiting Ricci flow
$\omega_\infty(\cdot)$
on the \'etale groupoid $(\R \rtimes \R) \times \R \times D$,
defined for $t \in [0, A]$.
Here the $(\R \rtimes \R)$-factor comes from the real factor in $\C$ and
represents the fact that the rescaled circle factor in $X$
collapses as $R \rightarrow \infty$.
The $\R$-factor is the imaginary factor in $\C$ and
represents the radial direction on $X$. 
The $D$-factor is the divisor.
(Because $D$ is compact, the
choice of basepoint $\overline{x} \in D$ is irrelevant.)
This limiting Ricci flow will have bounded curvature.

From the bulging standard spatial asymptotics of $\omega_X(0)$,
on the unit space $\C \times D$ of the groupoid we have
\begin{equation} \label{5.11}
\omega_\infty(0) = 
\frac12 \left( \frac{N+1}{N} \right)^2
\left(
\sqrt{-1} du \wedge d\overline{u} +
N \omega_D(0)
\right).
\end{equation}
From the uniqueness of Ricci flow solutions on \'etale groupoids
with bounded curvature on compact time intervals \cite{Hilaire},
we conclude that
\begin{equation} \label{5.12}
\omega_\infty(t) = 
\frac12 \left( \frac{N+1}{N} \right)^2
\left(
\sqrt{-1} du \wedge d\overline{u} +
N \omega_D(t)
\right).
\end{equation}

To return to the proof of the proposition, if the conclusion of the
proposition is not true then there are some $\overline{x} \in D$,
$F_{\overline{x}} : \Delta^* \times \Delta^{n-1} \rightarrow 
\overline{X}$, $S < \infty$ and a sequence
$r_i \rightarrow 0$ with the property that even
after passing to any subsequence of the $r_i$'s, the 
Ricci flows
$r_i^{- \frac{2}{N}}
\alpha_{r_i}^* 
\widetilde{F}_{\overline{x}}^* \omega_X(r_i^{\frac{2}{N}} \cdot)$
do not converge smoothly to 
$\frac12 \left( \frac{N+1}{N} \right)^2
\left(
\sqrt{-1} du \wedge d\overline{u} +
N G_{\overline{x}}^* \omega_D(\cdot)
\right)$ on the spacetime region
$\{z \in \C : |z| < S\} 
\times \Delta^{n-1} \times [0,A]$.
However, this contradicts the fact there is a subsequence of the
$r_i$'s so that
the pointed Ricci flows
$\left\{ \left( r_i^{- \frac{2}{N}} \omega_X( r_i^{\frac{2}{N}} \cdot), 
{F}_{\overline{x}} ( e^{-\frac{r_i^2}{2}}, 0 ) \right)
\right\}_{i=1}^\infty$ smoothly converge in the pointed sense to the Ricci flow
$\frac12 \left( \frac{N+1}{N} \right)^2
\left(
\sqrt{-1} du \wedge d\overline{u} +
N \omega_D(\cdot)
\right)$ on the unit space $\C \times D$ of the \'etale groupoid
$(\R \rtimes \R) \times \R \times D$, for the time interval $[0, A]$.
\end{proof} 
{\em Proof of Theorem \ref{intro3} :} If the theorem is not true
then after replacing $\{x_i\}_{i=1}^\infty$ by a subsequence, we can
assume that no subsequence $\{x_{i_j}\}_{j=1}^\infty$ is such that
$\left\{ \left( X, (x_{i_j}, 0), \omega_{i_j}(\cdot) \right) 
\right\}_{j=1}^\infty$ has a limit given by (\ref{intro1.6}).

Thinking of $x_i$ as an element of $\overline{X}$, 
after passing to a subsequence, we can assume that
$\lim_{i \rightarrow \infty} \overline{x}_i = \overline{x}$ for some
$\overline{x} \in D$. Then for large $i$, we can find $(z_i, w_i) \in 
\Delta^* \times \Delta^{n-1}$ so that $x_i = F_{\overline{x}}(z_i, w_i)$,
with $\lim_{i \rightarrow \infty} |z_i| = 0$. 

Define $r_i \in \R^+$ by $e^{\: - \: \frac{r_i^2}{2}} = |z_i|$.
Then $\lim_{i \rightarrow \infty} r_i = \infty$ and
$\widetilde{F}_{\overline{x}}(\alpha_{r_i}(0), w_i) =
F_{\overline{x}}(|z_i|, w_i)$. Applying Proposition \ref{5.8}
with $N=n$ gives,
\begin{equation} 
\lim_{i \rightarrow \infty} \omega_{r_i}(\cdot) =  
\frac12 \left( \frac{n+1}{n} \right)^2
\left(
\sqrt{-1} du \wedge d\overline{u} +
n G_{\overline{x}}^* \omega_D(\cdot)
\right).
\end{equation}
This is a contradiction. \qed

\begin{remark} \label{5.13}
The only role of the Ricci curvature bound (\ref{5.9}) is to ensure
an appropriate 
biLipschitz condition between $\omega_X(x,t)$ and $\omega_X(x,0)$,
in order to apply Proposition \ref{A.1} of the appendix.  Other
curvature conditions imply this.  For example, given any continuous function
$f : [0, \infty) \rightarrow  [0, \infty)$, it would be enough to
assume that 
\begin{equation} \label{5.14}
|\Ric(x,t)| \le d_0(x_0, x)^{- \frac{2}{N+1}} \: 
f(d_0(x_0, x)^{- \frac{2}{N+1}} t).
\end{equation}
Equation (\ref{5.9}) is the special case when $f$ is a constant function.
\end{remark}

\begin{remark}
One can ask how generally the Ricci curvature assumption in
Proposition \ref{5.8} holds. It obviously holds if the initial
metric is Ricci-flat, as in the work of Tian and Yau \cite{tian-yau2}.
We expect that it also holds, at least, if the initial metric is a
perturbation of the Ricci-flat metric.
\end{remark}

\subsection{Bulging superstandard spatial asymptotics} \label{subsection4.3}

We now specialize to the case when $D$ is ample.
Let $h$ be a Hermitian metric on $L_D$ with positive curvature form.
Let $\omega_D$ be the restriction, to $D$, of the curvature form
associated to $h$.
As before, $\sigma$ is a holomorphic section of $L_D$ with zero-set $D$.
Let $L_D^1$ denote the unit circle bundle
of $L_D$.

\begin{defin} \label{5.15}
A K\"ahler metric $\omega_X$ on $X$ has {\em bulging superstandard spatial
asymptotics} associated to $h$
if it has  bulging
standard spatial asymptotics, associated to $(\omega_D, N)$, and
\begin{equation} \label{5.16}
\omega_X = 
\eta_{\overline{X}} +
 {\sqrt{-1}} \partial \overline{\partial}  \left( \frac{N+1}{2}
\left( \log |\sigma|_h^{-2} \right)^{\frac{N+1}{N}}
+ H \right), 
\end{equation}
where 
\begin{enumerate} [\upshape (i)]
\item $\eta_{\overline{X}}$ is a smooth closed $(1,1)$-form on
$\overline{X}$, and
\item $H \in C^\infty(X) \cap L^\infty(X)$. 
\end{enumerate}
\end{defin}

Note that in Definition \ref{5.15}, the choice of $h$ does matter.

\begin{proposition} \label{5.17}
If $\overline{X}$ admits a K\"ahler metric then $X$ admits a
complete K\"ahler metric with bulging superstandard spatial asymptotics.
\end{proposition}
\begin{proof}
Let $\omega_{\overline{X}}$ be a K\"ahler metric on $\overline{X}$. 
By assumption,
$-\sqrt{-1}\partial\bar{\partial} \left( \log  |\sigma|_h^2 \right)$
is a positive $(1,1)$-form on $\overline X$.
Put
\begin{equation} \label{5.18}
\omega_X =   \frac{N+1}{2} {\sqrt{-1}} \partial \overline{\partial} 
\left( \log |\sigma|_h^{-2} \right)^{\frac{N+1}{N}}.
\end{equation}

We first claim that $\omega_X$ is a
complete K\"ahler metric on $X$ with bulging standard spatial asymptotics.
To see this, we have
\begin{align} \label{5.19}
\omega_X 
 = &  \frac{(N+1)^2}{2N} 
\sqrt{-1} \left( \log |\sigma|_h^{-2} \right)^{\frac{1}{N}} 
\partial\overline{\partial} \left( \log |\sigma|_h^{-2} \right) + \\
&   \frac12 \left( \frac{N+1}{N} \right)^2
\left( \log |\sigma|_h^{-2} \right)^{\frac{1-N}{N}} 
\sqrt{-1}\partial \left( \log |\sigma|_h^{-2}\right) \wedge\overline{\partial} 
\left( \log |\sigma|_h^{-2} \right). \notag
\end{align}
Since 
$-\sqrt{-1}\partial\bar{\partial} \left( \log  |\sigma|_h^2 \right)$ 
is positive,
it follows that $\omega_X$ is a K\"ahler form. One can check 
that it is complete.

To see that $\omega_X$ has bulging standard spatial asymptotics,
given $\overline{x}\in D$, 
let $(z, w^1, \ldots, w^{n-1})$ be the local holomorphic coordinates for a neighborhood of $\overline{x}$ in $\overline{X}$ coming from $F_{\overline{x}}$.
In this coordinate system, $|\sigma|_h^2 = az\overline{z}$ for some smooth 
positive function $a(z, w^1, \ldots, w^{n-1})$. 
On $H \times \Delta^{n-1}$, we have
\begin{equation} \label{5.20}
\log |\sigma|_h^{-2} = \log \frac{1}{az\overline{z}} =
\log \left( \frac{1}{a e^{\sqrt{-1} u} e^{-\sqrt{-1} \overline{u}}} \right) =
\log a^{-1} - \sqrt{-1} u + \sqrt{-1} \overline{u}.
\end{equation}
Then
\begin{equation} \label{5.21}
\alpha_r^* \log |\sigma|_h^{-2} =
\alpha_r^*
\log a^{-1} - \sqrt{-1} ru + \sqrt{-1} r \overline{u} + r^2,
\end{equation}
\begin{equation} \label{5.22}
\alpha_r^* \partial \left( \log |\sigma|_h^{-2} \right) =
\partial \alpha_r^*
\log a^{-1} - \sqrt{-1} r du
\end{equation}
and
\begin{equation} \label{5.23}
\alpha_r^* \overline{\partial} \left( \log |\sigma|_h^{-2} \right) =
\overline{\partial} \alpha_r^*
\log a^{-1} + \sqrt{-1} r d\overline{u}.
\end{equation}
It follows that
\begin{equation} \label{5.24}
\lim_{r \rightarrow \infty} r^{- \frac{2}{N}} \alpha_r^* 
\widetilde{F}_{\overline{x}}^* \omega_X =
\frac12 \left( \frac{N+1}{N} \right)^2 \left( 
\sqrt{-1} du \wedge d\overline{u} + N
G_{\overline{x}}^* \omega_D \right).  
\end{equation}
Hence $\omega_X$ has bulging standard spatial asymptotics.

From (\ref{5.18}), 
it is now clear that $\omega_X$ also has bulging superstandard 
spatial asymptotics.
This proves the proposition.
\end{proof}

In the rest of this section, we
assume that $N = n$, the complex dimension of $X$,
The next proposition shows that the property of having bulging superstandard
spatial asymptotics is preserved under the K\"ahler-Ricci flow.

\begin{proposition} \label{5.25}
Suppose that $\omega_X(0)$ has 
bulging superstandard spatial asymptotics associated
to $h$.
Suppose that the K\"ahler-Ricci flow $\omega_X(t)$,
with initial K\"ahler metric $\omega_X(0)$, exists on a maximal time
interval $[0,T)$ in the sense of Theorem \ref{2.5}. Then for all
$t \in [0,T)$, $\omega_X(t)$ has bulging
superstandard spatial asymptotics,
associated to $h$.
\end{proposition}
\begin{proof}
Choose a Hermitian metric $h_{K_{\overline{X}} \otimes L_D}$
on $K_{\overline{X}} \otimes L_D$.
Along with $h$, we obtain a Hermitian metric 
$h_{K_{\overline{X}}}$ on $K_{\overline{X}}$.
Then
\begin{equation} \label{5.26}
\Ric (\omega_X(0))  =  - \sqrt{-1} F \left( 
h_{K_{\overline{X}} \otimes L_D} \right)
 - \sqrt{-1} \partial \overline{\partial} \left(
\log
\frac{
h_{K_{\overline{X}}} |\sigma|_{h}^2
}{
h_{K_X}
} \right)
\end{equation}
on $X$. 

Put $\eta^\prime_{\overline{X}} = - \sqrt{-1} F \left( 
h_{K_{\overline{X}} \otimes L_D} \right)$ and $H^\prime =
\log \frac{
h_{K_{\overline{X}}} |\sigma|_{h}^2
}{
h_{K_X}
}$. By equation (\ref{5.3}), the bulging
standard spatial asymptotics imply that $H^\prime \in C^\infty(X) \cap
L^\infty(X)$. This is the place where we use that $N=n$.

Recall the definition of $\omega_t$ from (\ref{2.2}). We can write
\begin{align} \label{5.27}
\omega_X(t) = & \omega_t + \sqrt{-1} \partial \overline{\partial} u(t) \\
= & \eta_{\overline{X}} - t \eta^\prime_{\overline{X}} + \sqrt{-1}
\partial \overline{\partial} \left( 
 \frac{n+1}{2}
\left( \log |\sigma|_h^{-2} \right)^{\frac{n+1}{n}}
+ H + t H^\prime + u(t) \right). \notag
\end{align}
The proposition follows.
\end{proof}

\begin{proposition} \label{5.28}
Suppose that $\omega_X(0)$ has bulging
superstandard spatial asymptotics
associated to $h$.
Let $\eta_{\overline{X}} \in \Omega^{(1,1)}(\overline{X})$ 
be a smooth representative of the cohomology class represented by the
closed current
\begin{equation} \label{5.29}
\omega_X(0) -
\sqrt{-1} \partial \overline{\partial} \left( \frac{n+1}{2}
\left( \log \left| \sigma \right|_{h}^{-2} \right)^{\frac{n+1}{n}} \right)
\end{equation}
on $\overline{X}$.
Let $\eta^\prime_{\overline{X}} \in \Omega^{(1,1)}(\overline{X})$ be
a smooth representative of 
$- 2 \pi [K_{\overline{X}} + D] \in \HH^{(1,1)}(\overline{X})$.
Let $T_3$ be the supremum (possibly infinite) of the numbers $T^\prime$
for which there is some
$f_{T^\prime} \in C^\infty(X) \cap L^\infty(X)$ with bounded
covariant derivatives (with respect to $\omega_X(0)$) so that
\begin{equation} \label{5.30}
\eta_{\overline{X}} - T^\prime \eta^\prime_{\overline{X}} + 
\sqrt{-1} \partial \overline{\partial} \left( 
\frac{n+1}{2}
\left( \log \left| \sigma \right|_{h}^{-2} \right)^{\frac{n+1}{n}}
 + f_{T^\prime} \right)
\end{equation}
is a K\"ahler form on $X$ which is biLipschitz to $\omega_X(0)$.
Then $T_3$ equals the numbers $T_1 = T_2$ of Theorem \ref{2.5}.
\end{proposition}
\begin{proof}
The proof is similar to that of Proposition \ref{3.28}.  We omit the details.
\end{proof}

\begin{corollary}
Suppose that $\omega_X(0)$ has bulging superstandard spatial
asymptotics.  If $[K_{\overline{X}} + D] \ge 0$ then the flow exists
for all positive time.
\end{corollary}
\begin{proof}
The proof is similar to that of Corollary \ref{3.38}.  We omit the details.
\end{proof}

\begin{remark}
In \cite[Section 4]{tian-yau2}, Tian and Yau showed that if $K_{\overline{X}} 
+ D = 0$ then there is a complete Ricci-flat K\"ahler metric on $X$.
\end{remark}

\section{Conical K\"ahler-Ricci flows} \label{section5}

This section deals with conical spatial asymptotics. 
We begin by defining asymptotically conical Riemannian metrics.
In
Subsection \ref{subsection5.1} we show that this property
is preserved under Ricci flow, with no change in the asymptotic cone.
In Subsection \ref{subsection5.2} we use pseudolocality to 
show that if the initial metric of an immortal solution
is asymptotically conical then
a parabolic blowdown limit always exists, defined on the complement
of the vertex in a cone. 

Passing to the K\"ahler case, in Subsection \ref{subsection5.3}
we define conical standard spatial asymptotics and 
show that this property is preserved under the
K\"ahler-Ricci flow.
Starting with Subsection \ref{subsection5.4}, we assume that $D$ is ample.
We introduce conical superstandard
spatial asymptotics. We show that if $\overline{X}$ admits a
K\"ahler metric then $X$ admits a metric with
conical superstandard spatial asymptotics.
We prove that having conical superstandard spatial
asymptotics is preserved under the K\"ahler-Ricci flow. 
We characterize the first singularity time, if there is one.

In Subsection \ref{subsection5.5} we consider an ample line
bundle $E$ over a compact complex manifold $D$. We show that
there is a unique formal asymptotic expansion for an
expanding K\"ahler soliton on the complement of the zero-section
of $E$, whose associated vector
field generates rescaling of the line bundle.
The soliton turns out to be
a gradient soliton. Given an asymptotic expansion for
a K\"ahler-Ricci flow on the complement of the zero-section, 
we show that its blowdown
limit is the expanding soliton. We apply this to the case
of conical asymptotics.

\subsection{Asymptotically conical metrics} \label{subsection5.1}

Let $(Y, g_Y)$ be a compact Riemannian manifold.
Let $CY$ denote the cone over $Y$, i.e.
$CY = ((0, \infty) \times Y) \cup \{\star\}$, with the metric
$g_{CY} = dR^2 + R^2 g_Y$ on $CY - \star$.

Let $X$ be a smooth manifold with a basepoint $x_0$.
Let $g_X$ be a complete Riemannian metric on $X$.
Given $0 < R_1 < R_2 < \infty$, put $A(R_1, R_2) =
\overline{B(x_0, R_2)} - B(x_0, R_1)$.
 
\begin{defin} \label{6.1}
We say that $(X, x_0, g_X)$
is {\em asymptotically conical}, with asymptotic cone $(CY, g_{CY})$, 
if there is a pointed Gromov-Hausdorff
limit 
$\lim_{\lambda \rightarrow \infty} \left( \frac{1}{\lambda} X, x_0, g \right)
= \left( CY, \star, g_{CY} \right)$ so that
for $0 < R_1 < R_2 < \infty$, there is a smooth limit
\begin{equation} \label{6.2}
\lim_{\lambda \rightarrow \infty}
\frac{1}{\lambda} A(\lambda R_1, \lambda R_2) = 
\overline{B(\star, R_2)} - B(\star, R_1).
\end{equation}
\end{defin}

In Definition \ref{6.1}, 
it is implicit that the same maps are used to define the
Gromov-Hausdorff limit and the smooth limits.
From the definition, $(X, g_X)$ automatically has quadratic
curvature decay. 

\begin{remark} \label{6.3}
If we instead assumed that $(X, g_X)$ has quadratic
curvature decay, $Y$ is a compact metric space whose Hausdorff dimension is
$\dim(X) - 1$,
and that there is a pointed Gromov-Hausdorff
limit 
$\lim_{\lambda \rightarrow \infty} \left( \frac{1}{\lambda} X, x_0, g \right)
= \left( CY, \star, g_{CY} \right)$, then we would conclude that
$Y$ is a smooth manifold with a $C^{1,\alpha}$-regular Riemannian metric $g_Y$.
For simplicity, in Definition \ref{6.1} 
we will just assume that $g_Y$ is smooth.
\end{remark}

\begin{proposition} \label{6.4}
Let $g_X(\cdot)$ be a Ricci flow that exists for $t \in [0, T)$ and
whose initial condition
$g_X(0)$ is asymptotically conical with asymptotic cone $CY$.
Suppose that $g_X(\cdot)$ has
complete time slices and bounded curvature on compact time intervals.
Then for all $t \in [0, T)$, the time slice $(X, g(t))$ 
is asymptotically conical with asymptotic cone $CY$.
\end{proposition}
\begin{proof}
Suppose that the conclusion of the proposition is not true. Then for some
$t \in [0, T)$, there is a sequence $\lambda_i \rightarrow \infty$ 
with the property that even 
after passing to any subsequence of the
$\lambda_i$'s, either 
\begin{enumerate} [\upshape (i)]
\item 
The sequence $\left\{ \left(
\frac{1}{\lambda_i} X, x_0, g(t) \right) \right\}_{i=1}^\infty$
does not have a pointed Gromov-Hausdorff limit isometric to $CY$, or
\item
There are some $0 < R_1 < R_2 < \infty$ so that
the sequence
$\left\{ \left( \frac{1}{\lambda_i} A(\lambda_i R_1, \lambda_i R_2)
\right) \right\}_{i=1}^\infty$
does not have a smooth limit.
\end{enumerate}

Let $\phi$ be a slight smoothing of the function $1 + d_0(\cdot, x_0)$. 
Because of the conical asymptotics, we can assume that the Hessian of $\phi$ 
is bounded in norm. Then $\phi$ is a distance-like function that satisfies
(\ref{4.3.5}). From Proposition \ref{4.5},
on the time interval $[0,t]$, the curvature of $X$ decays uniformly
as $O \left( d_0(x,x_0)^{- 2} \right)$. It follows from the
Ricci flow equation that at a point $x \in X$,
the metrics $g_X(t)$ and $g_X(0)$ are
$e^{\const t  d_0(x,x_0)^{-2}}$-biLipschitz to each other. Hence
$\lim_{i \rightarrow \infty} \left(
\frac{1}{\lambda_i} X, x_0, g(t) \right) = CY$ in the pointed
Gromov-Hausdorff topology, so we can assume that
(ii) holds.

Since $g_X(0)$ is asymptotically conical, 
given $k \ge 0$, we have $|\nabla^k \Rm|(x,0) =
O \left( d_0(x,x_0)^{-2-k} \right)$. From Proposition \ref{4.7},
$|\nabla^k \Rm|(x,t) =
O \left( d_0(x,x_0)^{-2-k} \right)$.
This implies that
there are uniform bounds for
$|\nabla^k \Rm|_{\frac{1}{\lambda_i^2} g_X(t)}$ on
$A \left( \frac12 R_1, 2 R_2 \right)
 \subset \left( \frac{1}{\lambda_i} X, x_0 \right)$.
After passing to a subsequence and using the noncollapsing, 
we can assume that there is smooth convergence as
$i \rightarrow \infty$ of the metrics
$\frac{1}{\lambda_i^2} g_X(t)$ on
$A \left( R_1, R_2
\right) \subset \left( \frac{1}{\lambda_i} X, x_0 \right)$.
This is a contradiction.
\end{proof}

\subsection{Parabolic blowdowns of asymptotically conical
Ricci flows} \label{subsection5.2}

We now use pseudolocality to show that we have the curvature bounds needed to
take a blowdown limit of an immortal Ricci flow solution whose
initial metric is asymptotically conical.

\begin{proposition} \label{6.5}
Suppose that the Ricci flow $g_X(\cdot)$ of Proposition
\ref{6.4} is defined for all
$t \ge 0$.
Then there are $R < \infty$ and $\epsilon > 0$ so that 
$|\Rm(x,t)| \le (\epsilon d_0(x,x_0))^{-2}$ whenever
$d_0(x,x_0) \ge R$ and $t \le \epsilon d_0(x,x_0)^2$.
\end{proposition}
\begin{proof}
From Definition \ref{6.1}, there are $C, R^\prime < \infty$ so that
$|\Rm(x,0)| \le C d_0(x,x_0)^{-2}$
whenever
$d_0(x,x_0) \ge R^\prime$. Hence we can find
$\alpha > 0$ so that whenever
$d_0(x,x_0) \ge 2 R^\prime$, we have
$|\Rm| \le (\alpha d_0(x,x_0))^{-2}$ on
$B_0(x, \alpha d_0(x,x_0))$. 
From \cite[Proposition 1]{Lu}, there is some
$\epsilon_0 > 0$ so that
$|\Rm(x,t)| \le (\epsilon_0 \alpha d_0(x,x_0))^{-2}$
whenever $d_0(x,x_0) \ge 2 R^\prime$ and
$t \le (\epsilon_0 \alpha d_0(x,x_0))^{2}$. After redefining
the constants, the proposition follows.
\end{proof}

\begin{proposition} \label{6.6}
There are $R < \infty$ and $\epsilon > 0$ such that
for each $k > 0$, there is some $C_k < \infty$ so that
$|\nabla^k \Rm| \le C_k d_0(x,x_0)^{-2-k}$
whenever $d_0(x,x_0) \ge R$ and $t \le \epsilon d_0(x,x_0)^2$.
\end{proposition}
\begin{proof}
Given $k > 0$, Proposition \ref{6.5} and Shi's local derivative estimates
imply that for any $\epsilon_k \in (0, \epsilon)$, there is a bound
$|\nabla^k \Rm| \le C^\prime_k(\epsilon_k) d_0(x,x_0)^{-2-k}$
whenever $d_0(x,x_0) \ge 2R$ and 
$\epsilon_k d_0(x,x_0)^2 \le t \le \epsilon d_0(x,x_0)^2$.
Now by the smooth convergence in Definition \ref{6.1},
we know that there is some $C_k^\prime < \infty$ so that
$|\nabla^k \Rm|(x,0) \le C_k^\prime d_0(x,x_0)^{-2-k}$
whenever $d_0(x,x_0) \ge R$. From the local derivative estimate in
\cite[Appendix D]{Kleiner-Lott}, there are some $\epsilon_k > 0$
and $C_k^{\prime \prime} < \infty$ so that 
$|\nabla^k \Rm| \le C^{\prime \prime}_k d_0(x,x_0)^{-2-k}$
whenever $d_0(x,x_0) \ge 2R$ and $t \le \epsilon_k d_0(x,x_0)^2$.
The proposition follows.
\end{proof}

\begin{proposition} \label{6.7}
Let $\epsilon$ be as in Proposition \ref{6.6}.
For any sequence $r_i \rightarrow \infty$, after passing to a subsequence
there is a smooth Ricci flow solution $g_\infty$ defined on
$\{(x,t) \in (CY - \star) \times [0, \infty)  : 
0 \le t \le \epsilon d(x,\star)^2 \}$ so that
\begin{enumerate} [\upshape (i)]
\item $g_\infty(0) = g_{CY}$ and 
\item $\lim_{i \rightarrow \infty} \frac{1}{r_i^2} g(r_i^2 t) =
g_\infty(t)$, with smooth convergence on annuli.
\end{enumerate}

More precisely, given $0 < R_1 < R_2 < \infty$, for large $i$ there
are smooth embeddings $\phi_{R_1, R_2, i} : (R_1, R_2) \times Y
\rightarrow X$ so that $\lim_{i \rightarrow \infty} \frac{1}{r_i^2} 
\phi_{R_1, R_2, i}^* g(r_i^2 t) = g_\infty(t)$, provided that
$t \le \epsilon R_2^2$.
\end{proposition}
\begin{proof}
This follows from Proposition \ref{6.6} and the proof of
Hamilton's compactness theorem \cite{Hamilton2}. In our case,
we apply a diagonal argument to annuli; c.f. \cite[Section 2]{Hamilton2}. 
Note that 
$g_\infty(0)$ is defined on the complement of the vertex in the
asymptotic cone of $(X, g_X(0))$. The asymptotic cone is $CY$.
\end{proof}

\begin{proposition} \label{6.8}
For any $t \ge 0$, $g_\infty(t)$ is asymptotically conical with 
asymptotic cone $CY$.
\end{proposition}
\begin{proof}
We first remark that the estimate
$|\nabla^k \Rm| \le C_k d_0(x,x_0)^{-2-k}$ on 
$g_X(0)$, for
$k \ge 0$, passes to $g_\infty$. Hence for any $t \ge 0$ and any
sequence $s_j \rightarrow
\infty$, after passing to a subsequence the limit
$\lim_{j \rightarrow \infty} \left( X, x_0, \frac{1}{s_j^2} g_\infty(t)
\right) = CY$ is smooth away from $\star \in CY$. 
At $x \in CY - \star$, the metrics $g_\infty(t)$ and $g_\infty(0)$ are
$e^{\const t d_{CY}(x,\star)^{-2}}$-biLipschitz equivalent.  It follows that
any asymptotic cone of $g_\infty(t)$ is isometric to the
asymptotic cone of $g_\infty(0)$, which is $CY$.
\end{proof}

\begin{example} \label{6.9}
Suppose that $(X, g_0)$ is an asymptotically conical  Ricci-flat manifold.
Its asymptotic cone is the Ricci-flat cone $CY$.
The Ricci flow starting from $g_0$ is the static Ricci flow $g_X(t) = g_0$.
The blowdown limit is the static Ricci flow
$g_\infty(t) = g_{CY}$. This is an expanding gradient soliton with
respect to the function $f = \frac{R^2}{4t}$.
\end{example}

\begin{example} \label{6.10}
Consider the metrics constructed in
\cite[Section 5]{FIK}. They live on a $n$-dimensional complex manifold 
$X$, $n \ge 2$, which
is a complex line bundle over $\C P^{n-1}$. Given
$k > n$ and $p \in \R^+$, there is a Ricci flow solution
whose metric on $X - \C P^{n-1}$ is the
$\Z_k$-quotient of the following metric on $\C^n$ :
\begin{align} \label{6.11}
g_X(t) = & \left\{ |z|^{-2+2p} B \left( 
4(t+t_0)|z|^{- 2p} \right)
\delta_{\alpha \overline{\beta}} + \right. \\ 
& \left. \left[ (p-1) |z|^{2p} B \left( 
4(t+t_0)|z|^{- 2p} \right)
- 4(t+t_0)p B^\prime \left( 4(t+t_0)|z|^{- 2p} \right) \right]
|z|^{-4} z^{\overline{\alpha}} z^\beta \right\} dz^\alpha \:
dz^{\overline{\beta}}. \notag 
\end{align}
Here $t_0 > 0$ and $B$ is a certain smooth function with $B(0) > 0$.
Going out the end of $X$ corresponds to taking $z \rightarrow \infty$.
For all $t \ge 0$, the metric $g_X(t)$ is asymptotically conical with
asymptotic cone $CY = \C^n/\Z_k$,
where the $\Z_k$ action on $\C^n$ is multiplication by scalars,
and the
metric on the asymptotic cone is
\begin{equation} \label{6.12}
g_{CY} = B(0) \left\{ |z|^{-2+2p} 
\delta_{\alpha \overline{\beta}} + 
(p-1) |z|^{2p}
|z|^{-4} z^{\overline{\alpha}} z^\beta \right\} dz^\alpha \:
dz^{\overline{\beta}}. 
\end{equation}

Given a sequence $r_i \rightarrow \infty$, put
$\phi_i(z) = r_i^{\frac{1}{p}} z$. Then the blowdown limit is
\begin{equation}
\lim_{i \rightarrow \infty} \frac{1}{r_i^2} \phi_i^* g_X(r_i^2 t) = 
g_\infty(t),
\end{equation}
where
\begin{align} \label{6.13}
g_\infty(t) = & \left\{ |z|^{-2+2p} B \left( 
4t|z|^{- 2p} \right)
\delta_{\alpha \overline{\beta}} + \right. \\ 
& \left. \left[ (p-1) |z|^{2p} B \left( 
4t|z|^{- 2p} \right)
- 4tp B^\prime \left( 4t|z|^{- 2p} \right) \right]
|z|^{-4} z^{\overline{\alpha}} z^\beta \right\} dz^\alpha \:
dz^{\overline{\beta}}. \notag 
\end{align}
For $t > 0$, this is an gradient expanding soliton solution.
At $t=0$, we have $g_\infty(0) = g_{CY}$.
\end{example}

\begin{example} \label{6.14}
With reference to Example \ref{6.9}, in the case $n=2$ 
let $g_0$ be a $k$-center Eguchi-Hanson metric
\cite{GH}. Then $CY = \C^2/\Z_k$,
where a generator of $\Z_k$ acts on $\C^2$ by
$\begin{pmatrix}
e^{\frac{2 \pi \sqrt{-1}}{k}} & 0 \\
0 & e^{\frac{- 2 \pi \sqrt{-1}}{k}}
\end{pmatrix},$
and $g_\infty(\cdot)$ is the static flow on $CY - \star$.
\end{example}

\begin{example}
In \cite{Schulze-Simon}, it was shown that if $g_X(0)$ has
nonnegative curvature operator then any blowdown limit is
a {\em smooth} gradient expanding soliton.
In \cite[Remark 7.3]{Cabezas-Rivas-Wilking}, this result was extended to
nonnegative complex sectional curvture.
In \cite{Chodosh-Fong}, it was shown that in the K\"ahler case,
if one assumes positive holomorphic bisectional curvature
then the soliton is $U(n)$-invariant.
\end{example}

\subsection{Conical standard spatial asymptotics}
\label{subsection5.3}

We now specialize to the K\"ahler case.
Let $\overline{X}$ be a compact connected $n$-dimensional complex manifold.
Let $D$ be a smooth effective divisor in $\overline{X}$.
Let $L_D$ be the holomorphic line bundle on $\overline{X}$ associated
to $D$. There is a holomorphic section $\sigma$ of $L_D$ with zero set
$D$, which is nondegenerate at $D$.  It is unique up to multiplication
by a nonzero complex number. 
Let $h$ be a Hermitian metric on $L_D$.
Let $\nabla$ be the corresponding Hermitian holomorphic connection.
Let $\nabla_D$ be its restriction to $D$.
Let $L_D^1$ denote the unit circle bundle
of $L_D$.

Suppose that $\overline{x} \in D$.
There are a neighborhood
$U$ of $\overline{x}$ in $\overline{X}$ and
a biholomorphic map
$F_{\overline{x}} : \Delta^n \rightarrow U$
so that
\begin{enumerate} [\upshape (i)]
\item $F_{\overline{x}}(0) = \overline{x}$, and
\item $F_{\overline{x}} \left( \Delta^* \times
\Delta^{n-1} \right) = U \cap X$.
\end{enumerate}

The map $G_{\overline{x}}$ on
$\Delta^{n-1}$, given by $G_{\overline{x}}(w) = 
F_{\overline{x}}(0,w)$, is a biholomorphic map from $\Delta^{n-1}$
to a neighborhood of $\overline{x}$ in $D$.
Let $z$ the local coordinate on $U$ coresponding to the first factor
in $\Delta^n$.

Given $r \in \R^+$, let $\alpha_{r} : \Delta^* \rightarrow \Delta^*$ be 
multiplication by $r$.
If $Z$ is an auxiliary space
then we also write $\alpha_{r}$ for $(\alpha_{r}, \Id) : 
\Delta^* \times Z \rightarrow \C \times Z$.

\begin{defin} \label{7.1}
Let $\omega_D$ be a K\"ahler metric
on $D$. Let $h_D$ be the restriction of $h$ to $L_D \big|_D$.
We identify
$z$ with a local multiple of $\sigma$. We say that 
$\omega_X$ has {\em conical standard spatial asymptotics} associated
to $(\omega_D, h_D)$ if for
every $\overline{x} \in D$ and every local parametrization
$F_{\overline{x}}$,
after possibly multiplying $z$ by a constant, we have
\begin{equation} \label{7.2}
\lim_{r \rightarrow 0} 
\frac{1}{r^2}
\alpha_{r}^* 
F_{\overline{x}}^* \omega_X =
\sqrt{-1} \frac{h_D(\nabla_D z \wedge \overline{\nabla_D z})}{|z|_h^4} +
\frac{G_{\overline{x}}^* \omega_D}{|z|_h^2}.
\end{equation}
The limit in (\ref{7.2}) means smooth convergence on
any subset $\{z \in \Delta^* : 0 < R_1 < |z| < R_2 < 1\} \times \Delta^{n-1}$
of $\Delta^* \times \Delta^{n-1}$. 
\end{defin}

Going out the end of $X$ corresponds to taking $z \rightarrow 0$.
One can check that 
the notion of conical standard spatial asymptotics in
Definition \ref{7.1} is consistent under change of local coordinate.

\begin{proposition} \label{7.3}
If $(X, \omega_X)$ has conical standard spatial asymptotics then it is
asymptotically conical in the sense of Definition \ref{6.1}. The asymptotic
cone is $CY$, where $Y$ is the unit circle bundle $L^1_D$.
\end{proposition}
\begin{proof}
In terms of the local coordinates $(w^1, \ldots, w^{n-1})$ on 
$\Delta^{n-1}$, there is a function $h(w)$ so that
$h_D(z,z) = h(w) |z|^2$. Then
\begin{align} \label{7.4}
\sqrt{-1} \frac{h_D(\nabla_D z \wedge \overline{\nabla_D z})}{|z|_h^4} +
\frac{G_{\overline{x}}^* \omega_D}{|z|_h^2} \: = \: &
\sqrt{-1} \frac{h (dz + z h^{-1} \partial_w h) \wedge 
(d\overline{z} + \overline{z} h^{-1} \partial_{\overline{w}} h)
}{(h|z|^2)^2} + \frac{\omega_D}{h|z|^2} \\
\: = \: &
\sqrt{-1} \frac{ \left( \frac{dz}{z} + \frac{\partial_w h}{h} \right) \wedge 
\left( \frac{d\overline{z}}{\overline{z}} + 
\frac{\partial_{\overline{w}}h}{h} \right)
}{h|z|^2} + \frac{\omega_D}{h|z|^2}. \notag
\end{align}

Write $z = \alpha e^{\sqrt{-1} \theta}$, so
\begin{equation} \label{7.5}
\frac{dz}{z} = \frac{d\alpha}{\alpha} + \sqrt{-1} d\theta.
\end{equation}
Define $R$ by 
\begin{equation} \label{7.6}
R^2 = \frac{1}{h |z|^2} = \frac{1}{h\alpha^2}.
\end{equation}
Then
\begin{equation} \label{7.7}
\frac{dR}{R} = - \frac{d\alpha}{\alpha} - 
\frac12 \frac{\partial_w h}{h} - \frac12 \frac{\partial_{\overline{w}} h}{h}.
\end{equation}
It follows that
\begin{equation} \label{7.8}
\frac{ \left( \frac{dz}{z} + \frac{\partial_w h}{h} \right)
\left( \frac{d\overline{z}}{\overline{z}} + 
\frac{\partial_{\overline{w} h}}{h} \right)
}{h|z|^2} + \frac{g_D}{h|z|^2} =
dR^2 + R^2 \left\{ \left[ d\theta + \frac{1}{2\sqrt{-1}} 
\left( \frac{\partial_w h}{h} - \frac{\partial_{\overline w} h}{h} \right)
\right]^2 + g_D \right\}.
\end{equation}
The right-hand side of (\ref{7.8}) is the conical metric on $CY - \star$,
where $Y = L^1_D$ has the Sasaki metric.
\end{proof}

\begin{proposition} \label{7.10}
Let $\omega_X(\cdot)$ be a K\"ahler-Ricci flow defined for $t \in [0, T)$,
with bounded curvature on compact time intervals.
Suppose that $\omega_X(0)$ has conical
standard spatial asymptotics associated to
$(\omega_D, h_D)$. Then for all $t \in [0, T)$, the metric
$\omega_X(t)$ has conical standard spatial asymptotics associated to
$(\omega_D, h_D)$. 
\end{proposition}

\begin{proof}
The proof is similar to that of Proposition \ref{5.7}.
We omit the details.
\end{proof}

\subsection{Conical superstandard spatial asymptotics} \label{subsec7.2}
\label{subsection5.4}

We now assume that $D$ is ample.
Let $h$ be a Hermitian metric on $L_D$ with positive curvature form.
Let $\omega_D$ be the restriction, to $D$, of the curvature form
associated to $h$.
As before, $\sigma$ is a holomorphic section of $L_D$ with zero-set $D$.

\begin{defin} \label{7.11}
A K\"ahler metric $\omega_X$ on $X$ has {\em conical superstandard spatial
asymptotics} associated to $h$ and a number $k \in \R$
if it has conical
standard spatial asymptotics (associated to $(\omega_D, h_D)$) and
\begin{equation} \label{7.9}
\omega_X = 
\eta_{\overline{X}} +
 {\sqrt{-1}} \partial \overline{\partial}  \left( 
|\sigma|_h^{-2} + k \log |\sigma|_h^{-2}
+ H \right), 
\end{equation}
where 
\begin{enumerate} [\upshape (i)]
\item $\eta_{\overline{X}}$ is a smooth closed $(1,1)$-form on
$\overline{X}$, and
\item $H \in C^\infty(X) \cap L^\infty(X)$. 
\end{enumerate}
\end{defin}

\begin{remark} \label{7.13}
In \cite{tian-yau3}, in order to construct K\"ahler-Einstein metrics,
a class ${\sqrt{-1}} \partial \overline{\partial}  
|\sigma|_h^{-2\alpha}$ of model metrics was considered.
These metrics are also asymptotically conical. In terms of the asymptotic
cone $CY$, the manifold $Y$ is again a circle bundle over $D$ and
the parameter $\alpha$ determines the length of the
circle fiber in $Y$. For simplicity, we only consider the case
$\alpha = 1$.  The discussion below can be easily extended to 
general $\alpha > 0$.

Specializing the results of \cite{tian-yau3} to the case when $\alpha = 1$,
they showed that if $D$ admits a K\"ahler metric $\omega_D$ with 
$\Ric(\omega_D) = n \omega_D$ then there is a complete Ricci-flat K\"ahler
metric on $X$ \cite{tian-yau3}.
\end{remark}

Note that $\sqrt{-1} \partial \overline{\partial} \log |\sigma|_h^{-2}$ is 
bounded with respect to $g_{\overline{X}}$.

\begin{proposition} \label{7.14}
If $\overline{X}$ admits a K\"ahler metric then $X$ admits a
complete K\"ahler metric with conical superstandard spatial asymptotics.
\end{proposition}
\begin{proof}
Let $\omega_{\overline{X}}$ be a K\"ahler metric on $\overline{X}$. 
By assumption,
$\sqrt{-1} 
F_h = -\sqrt{-1}\partial\bar{\partial} \left( \log  |\sigma|_h^2 \right)$
is a positive $(1,1)$-form on $\overline X$. Put
\begin{equation} \label{7.15}
\omega_X =  {\sqrt{-1}} \partial \overline{\partial} 
|\sigma|_h^{-2}
= \sqrt{-1} \frac{
h(\nabla \sigma \wedge \overline{\nabla \sigma})
}{
|\sigma|_h^4} + \frac{\sqrt{-1} F_h}{|\sigma|_h^2}. 
\end{equation}
Taking the local coordinate $z$ to be $\sigma$,
\begin{equation} \label{7.16}
\lim_{r \rightarrow 0} \frac{1}{r^2} \alpha_r^* F_{\overline{x}}^* \omega_X =
\sqrt{-1} \frac{
h_D(\nabla_D z \wedge \overline{\nabla_D z})
}{
|z|_h^4} + \frac{\omega_D}{|z|_h^2}
\end{equation}
Hence $\omega_X$ has conical standard spatial asymptotics.
It clearly also has conical superstandard spatial asymptotics.
\end{proof}

We now show that the property of having conical superstandard
spatial asymptotics is preserved under the K\"ahler-Ricci flow.

\begin{proposition} \label{7.17}
Suppose that $\omega_X(0)$ has 
conical superstandard spatial asymptotics associated
to $(h,k)$.
Suppose that the K\"ahler-Ricci flow $\omega_X(t)$,
with initial K\"ahler metric $\omega_X(0)$, exists on a maximal time
interval $[0,T)$ in the sense of Theorem \ref{2.5}. Then for all
$t \in [0,T)$, $\omega_X(t)$ has conical
superstandard spatial asymptotics,
associated to $(h, k + n t)$.
\end{proposition}
\begin{proof}
Choose a Hermitian metric $h_{K_{\overline{X}} \otimes L_D}$
on $K_{\overline{X}} \otimes L_D$.
Along with $h$, we obtain a Hermitian metric 
$h_{K_{\overline{X}}}$ on $K_{\overline{X}}$.
Then
\begin{equation} \label{7.18}
\Ric (\omega_X(0))  =  - \sqrt{-1} F \left( 
h_{K_{\overline{X}} \otimes L_D} \right)
 - \sqrt{-1} \partial \overline{\partial} \left(
n \log |\sigma|_{h}^{-2} +
\log
\frac{
h_{K_{\overline{X}}} |\sigma|_{h}^{2(n+1)}
}{
h_{K_X}
} \right)
\end{equation}
on $X$. 

Put $\eta^\prime_{\overline{X}} = - \sqrt{-1} F \left( 
h_{K_{\overline{X}} \otimes L_D} \right)$ and $H^\prime =
\log \frac{
h_{K_{\overline{X}}} |\sigma|_{h}^{2(n+1)}
}{
h_{K_X}
}$. By equation (\ref{7.2}), the conical
standard spatial asymptotics imply that $H^\prime \in C^\infty(X) \cap
L^\infty(X)$. 

Recall the definition of $\omega_t$ from (\ref{2.2}). We can write
\begin{align} \label{7.19}
\omega_X(t) \: = \: & \omega_t + \sqrt{-1} \partial \overline{\partial} u(t) \\
\: = \: & \eta_{\overline{X}} - t \eta^\prime_{\overline{X}} + \sqrt{-1}
\partial \overline{\partial} \left( 
|\sigma|_h^{-2} + (k + n t) 
\log |\sigma|_h^{-2}
+ H + t H^\prime + u(t) \right). \notag
\end{align}
The proposition follows.
\end{proof}

We now give a characterization of the first singularity time, if there
is one. The relevant ring of functions, for conical asymptotics, can
be characterized in the following way.

\begin{defin} \label{7.20}
The ring $C^\infty_{\cone}(X)$ consists of the smooth functions $f$ on
$X = \overline{X} - D$ so that for every $\overline{x} \in D$ and
every local parametrization $F_{\overline{x}}$, the pullback
$F_{\overline{x}}^* f \in C^\infty(\Delta^* \times \Delta^{n-1})$
is such that for any multi-index $(l_1, \overline{l}_1, \ldots,
l_{n}, \overline{l}_{n})$, the function
\begin{equation} \label{7.19.5}
\left(z^2  \frac{\partial}{\partial z} \right)^{l_1}
\left( \overline{z}^2 \frac{\partial}{\partial \overline{z}} 
\right)^{\overline{l}_1}
\left( z \frac{\partial}{\partial w^1} \right)^{l_{2}}
\left( \overline{z} \frac{\partial}{\partial \overline{w}^1} 
\right)^{\overline{l}_{2}}
\ldots
\left( z \frac{\partial}{\partial w^{n-1}} \right)^{l_n}
\left( \overline{z} \frac{\partial}{\partial \overline{w}^{n-1}} 
\right)^{\overline{l}_n}
F_{\overline{x}}^* f \notag
\end{equation}
is uniformly bounded.
\end{defin}

\begin{proposition} \label{7.21}
Suppose that $\omega_X(0)$ has conical
superstandard spatial asymptotics
associated to $(h,k)$.
Let $\eta_{\overline{X}} \in \Omega^{(1,1)}(\overline{X})$ 
be a smooth representative of the cohomology class represented by the
closed current
\begin{equation} \label{7.22}
\omega_X(0) -
\sqrt{-1} \partial \overline{\partial}
\left( |\sigma|_h^{-2} + k \log |\sigma|_h^{-2} \right)
\end{equation}
on $\overline{X}$.
Let $\eta^\prime_{\overline{X}} \in \Omega^{(1,1)}(\overline{X})$ be
a smooth representative of 
$- 2 \pi [K_{\overline{X}} + D] \in \HH^{(1,1)}(\overline{X})$.
Let $T_3$ be the supremum (possibly infinite) of the numbers $T^\prime$
for which there is some
$f_{T^\prime} \in C^\infty_{\cone}(X)$ so that
\begin{equation} \label{7.23}
\eta_{\overline{X}} - T^\prime \eta^\prime_{\overline{X}} + 
\sqrt{-1} \partial \overline{\partial} \left( 
|\sigma|_h^{-2} + (k+nT^\prime) \log |\sigma|_h^{-2} + f_{T^\prime} \right)
\end{equation}
is a K\"ahler form on $X$ which is biLipschitz to $\omega_X(0)$.
Then $T_3$ equals the numbers $T_1 = T_2$ of Theorem \ref{2.5}.
\end{proposition}
\begin{proof}
The proof is similar to that of Proposition \ref{3.28}.  We omit the details.
\end{proof}

\begin{corollary}
Suppose that $\omega_X(0)$ has conical superstandard spatial
asymptotics.  If $[K_{\overline{X}} + (n+1) D] \ge 0$ then the flow exists
for all positive time.
\end{corollary}
\begin{proof}
The proof is similar to that of Corollary \ref{3.38}.  We omit the details.
\end{proof}

\subsection{Formal asymptotics} \label{subsection5.5}

In this subsection we discuss the asymptotics of the K\"ahler-Ricci
flow on the complement of the zero-section in the total space of a line 
bundle.  We then apply this to the quasiprojective case, where the
relevant line bundle is the normal bundle to the divisor.

Let $\pi : E \rightarrow D$ be an ample holomorphic line bundle $E$ over a complex manifold $D$.  Let $h$ be a Hermitian metric on $E$ with curvature $2$-form $F$, so that the representative $\sqrt{-1} F$ of $2\pi c_1(E)$ is a K\"ahler form $\omega_D$. 
In terms of a local holomorphic section $\sigma$ of $E$, we have $\omega_D = - \sqrt{-1} \partial \overline{\partial} \log |\sigma|^2_h$.

There is a canonical section $S$ of the bundle $\pi^* E$ over $E$ so that for $e \in E$, we have $S(e) = e \in (\pi^*E)_e \cong E_{\pi(e)}$. Let $E'$ be the complement of the zero-section of $E$. Define $\omega_0 \in \Omega^2(E')$ by $\omega_0 = \sqrt{-1} \partial\overline{\partial} |S|^{-2}_{\pi^* h}$. In terms of a local
holomorphic trivialization $E \big|_U = U \times \C$ of $E$ over a coordinate chart $U$, let $\{w^\alpha\}_{\alpha = 1}^{n-1}$ denote coordinates on $U$ and let $z$ denote the coordinate of the $\C$-factor. Then $|S|^{2}_{\pi^* h} = h(w) z \overline{z}$ and
\begin{align} \label{extra1}
\omega_0 \: = \: & \sqrt{-1} \partial\overline{\partial} \left( \frac{1}{h z \overline{z}} \right)  \\
\: = \: & \frac{\omega_D}{hz \overline{z}} + \sqrt{-1}\frac{1}{(hz \overline{z})^2} h (dz +  z h^{-1} \partial_w h) \wedge(d\overline{z} +  \overline{z} h^{-1} \overline{\partial}_w h). \notag
\end{align} 
There is a universal constant $C = C(n)$ so that 
\begin{equation}
\omega_0^n =C h^{-n} (z \overline{z})^{-(n+1)} dz \wedge d\overline{z} \wedge \omega_D^{n-1}. 
\end{equation}
Writing the metric in terms of the local coordinates $\{w^1, \ldots, w^{n-1}, z\}$ as $g_{i \overline{j}}$,
the Ricci curvature of $\omega_0$ is
\begin{equation} \label{extra2}
\Ric(\omega_0) = 
- \sqrt{-1} \partial \overline{\partial} \log\det(g_{i \overline{j}}) 
=- \sqrt{-1} \partial \overline{\partial} 
\log \left( h^{-n}\det(g^D_{\alpha \overline{\beta}}) \right) = \pi^*(\Ric(\omega_D) - n \omega_D).
\end{equation}

Put $V = z \partial_z + \overline{z} \partial_{\overline{z}}$, which is 
globally defined on $E^\prime$. One can check that 
${\mathcal L}_V \omega_0 = - 2 \omega_0$ and 
${\mathcal L}_V \Ric(\omega_0) = 0$. 
We use the notion of formal weight with respect to ${\mathcal L}_V$, which is
the same as the grading in the Taylor series expansion of a
function 
in terms of $z$ and $\overline z$.
Note that going out the conical end of $E^\prime$ corresponds to
taking $z \rightarrow 0$.  An expansion in $z$ (and $\overline{z}$) is
effectively an expansion in inverse powers of the distance from the
basepoint.

The expanding soliton equation for $\omega$, with respect to the vector field $V$, is
\begin{equation} \label{7.25}
\Ric(\omega) + \frac12 {\mathcal L}_V \omega = - \omega.
\end{equation}

\begin{proposition} \label{7.26}
Given $\omega_0$,
put
$\omega = \omega_0 - \Ric(\omega_0) + \sqrt{-1}
\partial \overline{\partial} u$.
There is a unique asymptotic expansion
\begin{equation} \label{7.27}
u \sim \sum_{k>0} u_{(k)},
\end{equation}
with ${\mathcal L}_V u_{(k)} = k u_{(k)}$, so that $\omega$
formally satisfies (\ref{7.25}). 
\end{proposition}
\begin{proof}
We have
\begin{align} \label{7.28}
{\mathcal L}_V \omega = & - 2 \omega_0 + \sqrt{-1} \sum_{k>0} 
{\mathcal L}_V \partial \overline{\partial} u_{(k)} \\
= & - 2 \omega_0 + \sqrt{-1} \sum_{k>0} 
\partial \overline{\partial} {\mathcal L}_V u_{(k)} \notag \\
= & - 2 \omega_0 + \sqrt{-1} \sum_{k>0} 
k \partial \overline{\partial}  u_{(k)}. \notag
\end{align}
Substituting into (\ref{7.25}) gives
\begin{align} \label{7.29}
& \Ric\left(\omega_0 - \Ric(\omega_0) + 
\sqrt{-1} \partial \overline{\partial} u\right) - \omega_0 + 
\sqrt{-1}\sum_{k>0}  \frac{k}{2} \partial
\overline{\partial} u_{(k)} = \\
& - \omega_0 + \Ric(\omega_0) -
 \sqrt{-1}\sum_{k>0} \partial
\overline{\partial}  u_{(k)}, \notag  
\end{align}
or
\begin{align} \label{7.30}
& - \partial \overline{\partial} \log 
\frac{(\omega_0 - \Ric(\omega_0) + \partial
\overline{\partial} \sum_{k>0} u_{(k)})^n
}{\omega_0^n}
+  \sum_{k>0} \frac{k}{2} \partial
\overline{\partial} u_{(k)} = \\
& 
- \sum_{k>0} \partial
\overline{\partial}  u_{(k)}. \notag  
\end{align}
Hence it suffices to solve
\begin{equation} \label{7.31}
\log 
\frac{(\omega_0 - \Ric(\omega_0) + 
\sqrt{-1} \partial
\overline{\partial} \sum_{k>0} u_{(k)})^n
}{\omega_0^n} = 
\sum_{k>0} \left( \frac{k}{2}+1 \right)  u_{(k)}. 
\end{equation}
Equivalently,
\begin{equation} \label{7.32}
\Tr \log 
\left( I - \omega_0^{-1} \Ric(\omega_0) + 
\sqrt{-1} \omega_0^{-1} \partial
\overline{\partial} \sum_{k>0} u_{(k)} \right)
=
\sum_{k>0} \left( \frac{k}{2}+1 \right)  u_{(k)}. 
\end{equation}

The term $\omega_0^{-1} \Ric(\omega_0)$ has formal weight
$2$ with respect to ${\mathcal L}_V$ and
$\omega_0^{-1} \partial \overline{\partial} u_{(k)}$ has
formal weight $k+2$. We can expand the left-hand side of (\ref{7.32})
with respect to the ${\mathcal L}_V$-weighting, as
\begin{equation} \label{7.33}
- \Tr 
\left( \omega_0^{-1} \Ric(\omega_0) \right) + 
\left[ - \frac12 
\Tr \left( \left( \omega_0^{-1} \Ric(\omega_0) \right)^2 \right)
+ \Tr \left( \sqrt{-1} \omega_0^{-1} \partial \overline{\partial} u_{(2)}
\right) 
\right]  + \cdots.
\end{equation}

It is easy to see directly that $u_{(2k-1)}=0$ for $k=1, 2, \cdots$. Also
\begin{equation} \label{7.34}
u_{(2)} = - \frac12 \Tr 
\left( \omega_0^{-1} \Ric(\omega_0) \right). 
\end{equation}
For $k=2, \cdots$, the term of weight $2k$ on the left-hand side of 
(\ref{7.32}) 
can be expressed in terms of $u_{(2)}, \ldots, u_{(2k-2)}$. Equating it with the term
$\left( k+1 \right)  u_{(2k)}$ of weight $2k$ on the right-hand side determines $u_{(2k)}$ inductively in terms of $u_{(2)}, \ldots, u_{(2k-2)}$. For example,
\begin{equation} \label{7.35}
u_{(4)} = \frac13
\left[- \frac12 
\Tr \left( \left( \omega_0^{-1} \Ric(\omega_0) \right)^2 \right)
+ \Tr \left( \sqrt{-1} \omega_0^{-1} \partial \overline{\partial} u_{(2)}
\right) 
\right].
\end{equation}
That gives the existence and uniqueness. 
\end{proof}

\begin{proposition} \label{7.26.5}
The formal expanding soliton of Proposition \ref{7.26}
is a gradient expanding soliton.
\end{proposition}
\begin{proof}
Let $V = V^{(1,0)} + V^{(0,1)}$ be the splitting of $V$ into its
$(1,0)$ and $(0,1)$ components, i.e. $V^{(1,0)} = z \partial_z$.
We need to find a function $F$ so that
\begin{equation} \label{extra3}
i_{V^{(1,0)}} \omega \: = \: \sqrt{-1} \: \overline{\partial} F.
\end{equation}

We first claim that $\omega$ is invariant under the $U(1)$-action
given by multiplying $z$ by complex numbers of norm one.
To see this, note that $\omega_0$ and $\Ric(\omega_0)$ are
$U(1)$-invariant.  Then from the inductive procedure to
construct $u_{(k)}$ in the proof of Proposition \ref{7.26}, it
follows that $u_{(k)}$ is $U(1)$-invariant.  Hence $\omega$ is
$U(1)$-invariant.

From the $U(1)$-invariance,
\begin{equation}
i_{V^{(1,0)}} \partial u_{(k)} = i_{V^{(0,1)}}
\overline{\partial} u_{(k)}.
\end{equation}
We know that
\begin{equation}
{\mathcal L}_V u_{(k)} = i_V du_{(k)} =
i_{V^{(1,0)}} \partial u_{(k)} + i_{V^{(0,1)}} 
\overline{\partial} u_{(k)} = k u_{(k)},
\end{equation}
so
\begin{equation}
i_{V^{(1,0)}} \partial u_{(k)} = \frac{k}{2} u_{(k)}.
\end{equation}
Then
\begin{equation}
i_{V^{(1,0)}} \left( \partial \overline{\partial} u_{(k)} \right) =
\overline{\partial} \left( i_{V^{(1,0)}} \partial u_{(k)} \right)
= \frac{k}{2} \overline{\partial} u_{(k)}
\end{equation}

From the definition of $\omega_0$,
\begin{equation}
i_{V^{(1,0)}} \omega_0 \: = \:  
\sqrt{-1} i_{V^{(1,0)}} \partial \overline{\partial} |S|^{-2}_{\pi^*h} \: = \:
\sqrt{-1} \: 
\overline{\partial} \left( i_{V^{(1,0)}} \partial |S|^{-2}_{\pi^*h}
\right) \: = \:
\:\sqrt{-1} \: \overline{\partial} \left( - |S|^{-2}_{\pi^*h} \right).
\end{equation}
From (\ref{extra2}), we have $i_{V^{(1,0)}} \Ric(\omega_0) = 0$.
Hence (\ref{extra3}) is satisfied with
\begin{equation}
F= - |S|^{-2}_{\pi^*h} + \sum_{k>0} \frac{k}{2} u_{(k)}.
\end{equation}
This proves the proposition.
\end{proof}

Now consider the K\"ahler-Ricci flow on $E^\prime$.
Put $\omega(t) = \omega_0 - t \Ric(\omega_0) + \sqrt{-1} \partial\overline{\partial} u(t)$. The flow equation for the potential function $u$ is
\begin{equation} \label{7.37}
\frac{\partial u}{\partial t} = \log \frac{(\omega_0 - t \Ric(\omega_0) + \sqrt{-1} \partial\overline{\partial} u)^n}{\omega_0^n} =\Tr \log \left( I - t \omega_0^{-1} \Ric(\omega_0) + \sqrt{-1}\omega_0^{-1} \partial
\overline{\partial} u \right).
\end{equation}

\begin{proposition} \label{7.38}
There is a unique asymptotic expansion 
\begin{equation} \label{7.39}
u(t) \sim \sum_{k=0}^\infty u_{(k)}(t),
\end{equation}
where ${\mathcal L}_V u_{(k)} = k u_{(k)}$, so that $u(\cdot)$ formally
satisfies (\ref{7.37}). The blowdown limit 
$u^\infty(w,z,t) = \lim_{s \rightarrow \infty} s^{-2} u(w,s^{-1}z,s^2 t)$
exists and equals the Ricci flow generated by the gradient expanding
soliton of Proposition \ref{7.26}.
\end{proposition}
\begin{proof}
Substituting (\ref{7.39}) into (\ref{7.37}) 
and equating the terms of various weights gives
\begin{align} 
\frac{\partial u_{(0)}}{\partial t} = & 0, \\
\frac{\partial u_{(1)}}{\partial t} = & 0, \notag \\
\frac{\partial u_{(2)}}{\partial t} = & - t \Tr \left(
\omega_0^{-1} \Ric(\omega_0) \right) + \sqrt{-1}
\Tr \left( \omega_0^{-1} \partial
\overline{\partial} u_{(0)} \right), \notag \\
\frac{\partial u_{(3)}}{\partial t} = & 
\sqrt{-1}
\Tr \left( \omega_0^{-1} \partial
\overline{\partial} u_{(1)} \right), \notag \\
\frac{\partial u_{(4)}}{\partial t} = & 
- \frac12 t^2 \Tr \left( \left(
\omega_0^{-1} \Ric(\omega_0) \right)^2  \right)
+ \sqrt{-1} t \Tr \left[ \left( \omega_0^{-1} \Ric(\omega_0) \right)
\left( \omega_0^{-1} \partial
\overline{\partial} u_{(0)} \right) \right]
+ \notag \\
& \frac12 \Tr \left(
\omega_0^{-1} \partial
\overline{\partial} u_{(0)} \right)^2 + \sqrt{-1}
\Tr \left( \omega_0^{-1} \partial
\overline{\partial} u_{(2)} \right), \notag  \\
\vdots \notag
\end{align}
The solution is
\begin{align} 
u_{(0)} = & c_0, \\
u_{(1)} = & c_1, \notag \\
u_{(2)} = & - \frac12 t^2 \Tr \left(
\omega_0^{-1} \Ric(\omega_0) \right) + \sqrt{-1} t
\Tr \left( \omega_0^{-1} \partial
\overline{\partial} c_0 \right) + c_2, \notag \\
u_{(3)} = & 
\sqrt{-1} t
\Tr \left( \omega_0^{-1} \partial
\overline{\partial} c_1 \right) + c_3, \notag \\
u_{(4)} = & 
- \frac16 t^3 \left[ \Tr \left( \left(
\omega_0^{-1} \Ric(\omega_0) \right)^2 \right) + \sqrt{-1}
\Tr \left( \omega_0^{-1} \partial
\overline{\partial} \Tr (\omega_0^{-1} \Ric(\omega_0)) \right)
\right] + \ldots + c_4, \notag  \\
\vdots \notag
\end{align}
where ${\mathcal L}_V c_k = k c_k$.

If $u^\infty(w,z,t) = \lim_{s \rightarrow \infty} s^{-2}u \left( w,s^{-1}z,s^2t \right)$, then $u^\infty$ has an asymptotic expansion
\begin{equation} \label{7.40}
u^\infty(t) \sim \sum_{k=1}^n u^\infty_{(k)}(t),
\end{equation}
with
\begin{align} \label{7.41}
u^\infty_{(0)} = & 0, \\
u^\infty_{(1)} = & 0, \notag \\
u^\infty_{(2)} = & - \frac12 t^2 \Tr \left(
\omega_0^{-1} \Ric(\omega_0) \right), \notag \\
u^\infty_{(3)} = &  0, \notag \\
u^\infty_{(4)} = & 
- \frac16 t^3 \left[ \Tr \left( \left(
\omega_0^{-1} \Ric(\omega_0) \right)^2 \right) + \sqrt{-1}
\Tr \left( \omega_0^{-1} \partial
\overline{\partial} \Tr (\omega_0^{-1} \Ric(\omega_0)) \right)
\right], \notag  \\
\vdots \notag
\end{align}
Note that the construction of $u^\infty(w, z, t)$ amounts to
keeping only the terms in $u$ with the highest power of $t$, 
i.e. $t^{\frac{k}{2}+1}$ for $u_{(k)}$. That is, we remove 
the terms involving the $c_i$'s. Then one sees that 
$\omega_0 - \Ric(\omega_0) + \sqrt{-1} \partial \overline{\partial} 
u^\infty(1)$ is the formal gradient expanding soliton 
of Proposition \ref{7.26}. Hence
$\omega_0 - \Ric(\omega_0) + \sqrt{-1} \partial \overline{\partial} 
u^\infty(t)$ is the time-$t$ solution for the flow of the
expanding soliton.
\end{proof}

\begin{proposition}
In the setting of Subsection \ref{subsection5.3},
suppose that $(X, \omega_X(\cdot))$ is a K\"ahler-Ricci flow with
conical standard spatial asymptotics, that exists on the
time interval $[0, \infty)$. 
Suppose that there is an asymptotic expansion
\begin{equation}
\omega_X(t) \sim \omega_0 - t \Ric(\omega_0) + \sqrt{-1} \sum_{k=0}^\infty
\partial \overline{\partial} u_{(k)}(t),
\end{equation}
where ${\mathcal L}_V u_{(k)} = k u_{(k)}$, with $V = z \partial_z +
\overline{z} \partial_{\overline{z}}$.
Let $(CY - \star, \omega_{X_\infty}(\cdot))$
be a parabolic blowdown limit of $\omega_X(\cdot)$. Then the
asymptotic expansion of $\omega_{X_\infty}(\cdot)$ is the
formal gradient expanding soliton of Proposition \ref{7.26}.
\end{proposition}
\begin{proof}
Let
$E$ be the restriction of $L_D$ to $D$, i.e. the normal bundle of $D$ in 
$\overline X$. Then the proposition follows from Proposition \ref{7.38}.
\end{proof}

We have now proved Theorem \ref{intro4}.

\begin{example}
With reference to Example \ref{6.10}, consider the case when $p =1$.
The asymptotic cone is flat.  The corresponding formal expanding
soliton is also flat. This is consistent with the explicit solution
of Example \ref{6.10}, which
approaches the asymptotic cone exponentially fast.  Note that in this
case, the parabolic blowdown limit is also an expanding soliton, which
differs from the formal expanding soliton of its asymptotic expansion.

If we took $p \neq 1$ in Example \ref{6.10} 
(see Remark \ref{7.13}) then the asymptotic cone, and also the
formal expanding soliton, would be nonflat.
\end{example}

\appendix

\section{Local curvature estimates in K\"ahler-Ricci flow} \label{localapp}

In this section we prove a curvature estimate for the K\"ahler-Ricci flow.
The assumptions are that the curvature and its first covariant 
derivative are bounded on an initial ball, and that on the
given time interval,
the metric on the ball is uniformly biLipschitz to the initial metric.

\begin{proposition} \label{A.1}
Let $(X, p, \omega(\cdot))$ be a pointed K\"ahler-Ricci flow on a complex 
$n$-dimensional manifold $X$, defined on a time interval $[0, T]$, with 
possibly incomplete time slices.  Given $C_1 < \infty$, there is some
$C_2 = C_2(C_1,n) < \infty$ with the following property.  If $r > 0$, 
suppose that the time-zero ball $B_0(p, r)$ has compact closure in $X$, 
and at time zero we have
\begin{enumerate} [\upshape (i)]
\item $|\Rm| \le C_1 r^{-2}$ on $B_0(p, r)$, and
\item $|\nabla \Rm| \le C_1 r^{-3}$ on $B_0(p, r)$.
\end{enumerate}
Furthermore, assume that for all $t \in [0,T]$,
\begin{equation} \label{A.2}
C_1^{-1} \omega(0) \le \omega(t) \le C_1 \omega(0). 
\end{equation}
Then $|\Rm(x,t)| \le C_2 r^{-2}$ for all $x \in B_0 \left(
p, \frac14 r \right)$ and $t \in [0, T]$.
\end{proposition}

\begin{remark}
In \cite{sherman-weinkove}, 
Sherman and Weinkove prove a related result in which
$C_2$ depends, in an unspecified way, on $\omega(0)$. Since we will
apply Proposition \ref{A.1} to a compactness theorem, we need a
uniformity result for $C_2$.
\end{remark}

\begin{remark}
The point of Proposition \ref{A.1} is that $C_2$ does not depend on $T$.
We do not know whether the analog of Proposition \ref{A.1} holds for 
non-K\"ahler Ricci flows.
\end{remark}

\begin{proof}
After rescaling, we can assume that $r=1$.

In the following, we use $\widetilde\cdot$ for notations including $g$, 
$|\cdot|$, $\nabla$, $\Rm$ and $\Delta$ to indicate if they are 
with respect 
to the initial metric $\omega(0)$. Otherwise, they are with respect to 
the evolving flow metric $\omega(\cdot)$. Indices are in terms of any 
local holomorphic coordinates on $X$. We let $\langle \cdot, \cdot \rangle$ 
denote a Euclidean inner product and $(\cdot, \cdot)$ denote a Hermitian 
inner product. We let $\nabla$ denote the $(1,0)$ component of the 
covariant derivative, and $\overline{\nabla}$ denote a $(0,1)$ component 
of the covariant derivative. For example, if $f_1$ and $f_2$ are smooth 
real functions then $(\nabla f_1, \nabla f_2) = g^{\bar j i}
\partial_i f_1 \partial_{\bar j} f_2$,
$\langle df_1, df_2 \rangle = \re (\nabla f_1, \nabla f_2)$ and
$\langle df_1, df_1 \rangle = |\nabla f_1|^2 = 
| \overline{\nabla} f_1 |^2$.

The letter $C$ will denote a positive constant that can depend on
$C_1$ and $n$, but does not depend in any other way on $\omega(0)$.
The value of $C$ is allowed to change from place to place.

Define the tensor $\Psi$ by 
\begin{equation} \label{A.5}
\Psi_{ij}^k = (\nabla - \widetilde\nabla)_{ij}^k =
g^{\bar l k}\widetilde\nabla_ig_{j\bar l}. 
\end{equation} 
In view of (\ref{A.2}), a bound on $\Psi$ is equivalent to a first derivative
bound on $\omega(t)$, relative to $\omega(0)$.
Put $S = |\Psi|^2$.
From \cite[Proposition 2.8]{song-weinkove},
\begin{equation}
\(\frac{\p}{\p t}-\Delta\)S=-|\overline\nabla\Psi|^2-|\nabla\Psi|^2
- 2 \re \left( g^{\bar j i} g^{\bar q p} g_{k \bar l} 
\nabla^{\bar b} \widetilde{R}_{i \bar b p}^{\: \: \: \: \: k}
\overline{\Psi^l_{jq}} \right).
\end{equation}
In view of (\ref{A.2}) and (\ref{A.5}),
\begin{equation}
- 2 \re \left( g^{\bar j i} g^{\bar q p} g_{k \bar l} 
\nabla^{\bar b} \widetilde{R}_{i \bar b p}^{\: \: \: \: \: k}
\overline{\Psi^l_{jq}} \right)
\le C|\widetilde\nabla\widetilde\Rm|\cdot|\Psi|+C|\widetilde\Rm|\cdot|\Psi|^2.
\end{equation}
Then using assumptions (i) and (ii) of the proposition,
\begin{equation}
\(\frac{\p}{\p t}-\Delta\)S\le -|\overline\nabla\Psi|^2-|\nabla\Psi|^2+C\cdot S+C.
\end{equation}
From \cite[(2.43)]{song-weinkove},
\begin{equation}
\nabla_{\bar b} \Psi^k_{ip} = \widetilde{R}_{i \bar b p}^{\: \: \: \: \: k} -
{R}_{i \bar b p}^{\: \: \: \: \: k}.
\end{equation}
Hence
\begin{equation} \label{A.10}
\(\frac{\p}{\p t}-\Delta\)S \le
-|\overline\nabla\Psi|^2+C\cdot S+C \le
 - \frac12 |\Rm|^2+C\cdot S+C.
\end{equation}

We will need a cutoff function.
Let $\Phi : [0, \infty) \rightarrow [0, \infty)$ be a smooth nonincreasing
function so that $\Phi \Big|_{\left[0,\frac12 \right]} = 1$ and 
$\Phi \Big|_{[1,\infty]} = 0$.
We can and will assume that where $\Phi \neq 0$,
\begin{equation} \label{A.11}
\frac{(\Phi^\prime)^2}{\Phi} = 4 
\left( \left( \Phi^{\frac12} \right)^\prime \right)^2 \le C.
\end{equation}
Put
\begin{equation}
\phi(x)=\Phi(d_0(x, p)),
\end{equation}
a time-independent Lipschitz function.
Let $d_p$ denote the time-zero distance function from $p$, i.e.
$d_p(x) = d_0(x,p)$
At time zero, 
\begin{equation} 
|\widetilde{\nabla} \phi|^2 = \widetilde{g}^{\bar j i} 
\partial_i \phi \partial_{\bar j} \phi = (\Phi^\prime)^2 \circ d_p.
\end{equation}
At time $t$,
$|\nabla \phi|^2 = g^{\bar j i} \partial_i \phi \partial_{\bar j} \phi$.
Then by (\ref{A.2}), $|\nabla \phi| \le C$ on $B_0(p,1) \times [0,T]$.

Next, the $(1,1)$-component of $\Hess(\phi)$ is given by
\begin{equation} \label{A.13}
\sqrt{-1} \partial \overline{\partial} \phi =
\sqrt{-1} \left( \Phi^{\prime \prime}\circ d_p \right) 
\partial d_p \wedge
\overline{\partial} d_p + \sqrt{-1} \left( \Phi^\prime \circ d_p \right)
\partial \overline{\partial} d_p.
\end{equation}
Assumption (i) of the proposition 
and Hessian comparison imply there is an estimate
\begin{equation}
\sqrt{-1} \partial \overline{\partial} d_p \le  C \omega(0)
\end{equation}
in the barrier sense.
As $\Phi^\prime \le 0$, using (\ref{A.2}) and (\ref{A.13}) we obtain
\begin{equation}
\triangle \phi = g^{\bar j i} \partial_i \partial_{\bar j} \phi \ge 
- C.
\end{equation}

Then on $B_0(p, 1)\times [0, T]$, we have
\begin{align}
&\(\frac{\p}{\p t}-\Delta\)(\phi\cdot S) \\
&= \phi\(\frac{\p}{\p t}-\Delta\)S-S\Delta\phi-2 \re (\nabla\phi, \nabla S) \notag \\
&\le \phi\(-|\overline\nabla\Psi|^2-|\nabla\Psi|^2+C\cdot S+C\)+C \cdot S+
2|\nabla\phi|\cdot|\nabla S|  \notag \\
&\le\phi\(-|\overline\nabla\Psi|^2-|\nabla\Psi|^2+C\cdot S+C\)+C \cdot S+
4|\nabla\phi|\cdot(|\nabla \Psi|+|\overline\nabla\Psi|)\cdot|\Psi| \notag
\end{align}

Let $\epsilon > 0$ be small enough that $\epsilon|\nabla \phi|^2-\phi\le 0$, 
which is possible from (\ref{A.11}). Since
\begin{align}
& \phi\(-|\overline\nabla\Psi|^2-|\nabla\Psi|^2+C\cdot S+C\)+
C \cdot S+4|\nabla\phi|\cdot(|\nabla \Psi|+|\overline\nabla\Psi|)\cdot|\Psi| \\
&\le \phi\(-|\overline\nabla\Psi|^2-|\nabla\Psi|^2+C\cdot S+C\)+
C \cdot S+\epsilon |\nabla\phi|^2\cdot(|\nabla \Psi|^2+|\overline\nabla\Psi|^2)+C(\epsilon)\cdot S, \notag
\end{align}
we conclude that
\begin{equation}
\(\frac{\p}{\p t}-\Delta\)(\phi\cdot S)\le C\cdot S+C.
\end{equation}

From assumption (i) of the proposition, (\ref{A.2}) and 
\cite[(2.26) and (2.27)]{song-weinkove}, we have
\begin{equation}
\(\frac{\p}{\p t}-\Delta\) \Tr \left( \omega_0^{-1} \omega \right)
\le C-C\cdot S.
\end{equation}
Choosing $A$ large enough, we have the following differential inequality on
$B_0(p, 1) \times [0, T]$:
\begin{equation} \label{A.21}
 \( \frac{\p}{\p t}-\Delta\) \left(\phi\cdot S+A \Tr \left( \omega_0^{-1} \omega
\right) \right) \le -S+C.
\end{equation}
 
On $\partial B_0(p, 1) \times [0, T]$, the function $\phi$ vanishes. On
$B_0(p,1) \times \{0\}$, the function $S$ vanishes. 
By (\ref{A.2}), 
\begin{equation}
A \: |\Tr \left( \omega_0^{-1} \omega \right)| \le C^\prime
\end{equation} for some
constant $C^\prime < \infty$. In particular,
\begin{equation}
\phi\cdot S+A \Tr \left( \omega_0^{-1} \omega
\right) \le C^\prime
\end{equation} on the parabolic boundary 
$(\partial B_0(p, 1) \times [0, T]) \cup (B_0(p,1) \times \{0\})$.

With reference to the constant $C$ of the right-hand side of
(\ref{A.21}), suppose that
\begin{equation} 
\phi(x,t) \cdot S(x,t) +A \Tr \left( \omega_0^{-1} \omega
\right)(x,t) \ge C + C^\prime
\end{equation} for some $(x,t)$. Then $\phi(x,t) \cdot S(x,t) \ge C$, so 
$S(x,t) \ge C$, so $-S(x,t) + C \le 0$. 
The reasoning in the proof of the parabolic maximum principle,
applied to (\ref{A.21}), now implies that 
\begin{equation}
\phi\cdot S+A \Tr \left( \omega_0^{-1} \omega \right)
 \le C + C^\prime
\end{equation}
on $B_0(p,1) \times [0, T]$. Reverting to the use of $C$ to denote a generic
positive constant, since
$\phi$ is one on $B_0\left( p, \frac12 \right) \times [0,T]$, 
we conclude that 
\begin{equation} \label{A.23}
S\le C
\end{equation} 
on $B_0\left( p, \frac12 \right) \times [0,T]$,
independent of $T$.

Hereafter we work on $B_0\left( p, \frac12 \right) \times [0,T]$.
From (\ref{A.10}) and (\ref{A.23}),
\begin{equation}
\(\frac{\p}{\p t}-\Delta\)S\le -\frac{1}{2}|\Rm|^2+C. 
\end{equation}
From \cite[(2.58)]{song-weinkove},
\begin{equation}
\(\frac{\p}{\p t}-\Delta\)|\Rm|^2\le -|\nabla\Rm|^2-|\overline\nabla\Rm|^2+C|\Rm|^3. 
\end{equation}

Put $H=\frac{|\Rm|^2}{(C-S)^\frac12}$
where $C$ is large enough so that $C-S\ge1$ in 
$B_0 \left( p, \frac12 \right)\times [0, T]$. Since
$1 \le C-S \le C$, a bound on $H$ is equivalent to a bound on $|\Rm|^2$.
Then
\begin{align} \label{A.30}
&\(\frac{\p}{\p t}-\Delta\)H \\
&= \frac{1}{(C-S)^\frac12}\(\frac{\p}{\p t}-\Delta\)
|\Rm|^2+|\Rm|^2\(\frac{\p}{\p t}-\Delta\)\frac{1}{(C-S)^\frac12} - 
2 \re \left( \nabla |\Rm|^2, \nabla \frac{1}{(C-S)^\frac12} \right)
\notag \\
&= \frac{1}{(C-S)^\frac12}\(\frac{\p}{\p t}-\Delta\)|\Rm|^2+
\frac12 \frac{|\Rm|^2}{(C-S)^\frac32}\(\frac{\p}{\p t}-\Delta\)S - \notag \\
&\,\,\,\,\,\,\,\,\,\, \frac38 \frac{|\Rm|^2(|\nabla S|^2+|\overline\nabla S|^2)}{(C-S)^\frac52}
- \frac{\re \left( \nabla |\Rm|^2, \nabla S \right)}{(C-S)^\frac32} 
 \notag \\
&\le \frac{1}{(C-S)^\frac12}\(-|\nabla\Rm|^2-|\overline\nabla\Rm|^2+C|\Rm|^3\)+ 
\frac12\frac{|\Rm|^2}{(C-S)^\frac32}\(-\frac{1}{2}|\Rm|^2+C\) - \notag \\
&\,\,\,\,\,\,\,\,\,\, \frac38 \frac{|\Rm|^2(|\nabla S|^2+|\overline\nabla S|^2)}{(C-S)^\frac52}
+ \frac{\sqrt{|\nabla \Rm|^2 + |\overline{\nabla} \Rm|^2}
\cdot |\Rm| \cdot  
\sqrt{|\nabla S|^2 + |\overline{\nabla} S|^2}}{(C-S)^\frac32} \notag \\
&= - \frac{11}{36} \:  \frac{|\nabla\Rm|^2+|\overline\nabla\Rm|^2}{(C-S)^\frac12}
- \frac{3}{200} 
\frac{|\Rm|^2(|\nabla S|^2+|\overline\nabla S|^2)}{(C-S)^\frac52} \: - \notag \\
&\,\,\,\,\,\,\,\,\,\, \left( \frac56
\frac{\sqrt{|\nabla \Rm|^2 + |\overline{\nabla} \Rm|^2}}{(C-S)^\frac14}
- \frac35 \frac{|\Rm| \cdot  
\sqrt{|\nabla S|^2 + |\overline{\nabla} S|^2}}{(C-S)^\frac54} \right)^2 
+ \notag \\
&\,\,\,\,\,\,\,\,\,\,\,\,\,\,\, \frac12
\frac{|\Rm|^2}{(C-S)^\frac32}\(-\frac{1}{2}|\Rm|^2 + 2 |\Rm| (C-S) +C\)
 \notag \notag \\
&\le - \frac{1}{100} \:  \frac{|\nabla\Rm|^2+|\overline\nabla\Rm|^2}{(C-S)^\frac12}
- \frac{1}{100} 
\frac{|\Rm|^2(|\nabla S|^2+|\overline\nabla S|^2)}{(C-S)^\frac52} \: 
- CH^2 + C. \notag
\end{align}

Define $\widehat{\phi} \in C(X)$ by
\begin{equation}
\widehat{\phi}(x) = \Phi(2 d_0(x,p)),
\end{equation}
so $\supp(\widehat{\phi}) \subset \overline{B_0 \left( p, \frac12 \right)}$. 
Then
\begin{align} \label{A.32}
 \(\frac{\p}{\p t}-\Delta\)(\widehat{\phi}\cdot H)
& =  \widehat{\phi}\(\frac{\p}{\p t}-\Delta\)H-H\Delta\widehat{\phi}-2\re(\nabla\widehat{\phi}, \nabla H) \\
&\le \widehat{\phi}\(\frac{\p}{\p t}-\Delta\)H+C \cdot H+
2 |\nabla \widehat{\phi}| \cdot|\nabla H|. \notag 
\end{align}
Now
\begin{equation}
\nabla H = \frac{\nabla |\Rm|^2}{(C-S)^\frac12} + \frac12
\frac{|\Rm|^2 \nabla S}{(C-S)^\frac32},
\end{equation}
so
\begin{align} \label{A.34}
& 2|\nabla \widehat\phi|\cdot|\nabla H| \\
&\le 2|\nabla \widehat\phi|\cdot\left(\sqrt{2} \frac{|\Rm| \sqrt{|\nabla \Rm|^2 + |\overline{\nabla} \Rm|^2}}{(C-S)^\frac12} 
+ \frac{1}{2\sqrt{2}}\frac{|\Rm|^2 \sqrt{|\nabla S|^2 + |\overline{\nabla} S|^2}}{(C-S)^\frac32}\right) \notag \\
&\le  \frac{\epsilon}{100} |\nabla \widehat\phi|^2
\left( \frac{|\nabla\Rm|^2+|\overline\nabla\Rm|^2}{(C-S)^\frac12}+
\frac{|\Rm|^2(|\nabla S|^2+|\overline\nabla S|^2)}{(C-S)^\frac52} \right) +
\notag \\
& \: \: \: \: \: \: \frac{100}{\epsilon}
\left( \frac{2 |\Rm|^2}{(C-S)^\frac12} + \frac{|\Rm|^2}{8(C-S)^\frac12} 
\right).
\notag
\end{align}
for any $\epsilon > 0$. From (\ref{A.11}), we can choose
$\epsilon$ so that $\epsilon |\nabla \widehat{\phi}|^2 - \widehat{\phi} \le 0$.
Using (\ref{A.30}), (\ref{A.32}) and (\ref{A.34}), we arrive at 
\begin{align}
\( \frac{\p}{\p t}-\Delta \)(\widehat{\phi}\cdot H) \le \widehat{\phi} \left(- CH^2 + C \right) + C \cdot H \le C |\Rm|^2 + C. 
\end{align}

For large $B < \infty$, equation (\ref{A.10}) now gives
\begin{equation}
 \(\frac{\p}{\p t}-\Delta\)(\widehat{\phi}\cdot H+B\cdot S)\le
-|\Rm|^2+C. 
\end{equation}
Using the parabolic maximum principle as in the proof of
(\ref{A.23}), we conclude that
\begin{equation}
 |\Rm|\le C 
\end{equation}
on $B_0 \left( p, \frac14 \right) \times [0, T]$.
This proves the proposition.
\end{proof}

We now extend the preceding proposition to include higher
derivative curvature bounds.

\begin{proposition} \label{A.38}
Let $(X, p, \omega(\cdot))$ be a pointed
K\"ahler-Ricci flow on a complex $n$-dimensional manifold
$X$, defined on a time interval $[0, T]$, with possibly incomplete 
time slices.  Given $l \ge 1$ and 
$\widetilde{C}_l < \infty$, there is some
$\widehat{C}_l = \widehat{C}_l(\widetilde{C}_l,n) < \infty$ 
with the following property.  If $r > 0$, 
suppose that the time-zero ball $B_0(p, r)$ has compact closure in $X$, 
and at time zero, for all $0 \le k \le l$ we have
\begin{equation}
\left| \nabla^k \Rm \right| \le \frac{\widetilde{C}_l}{r^{k+2}}
\end{equation}
on $B_0(p, r)$.
Further assume that for all $t \in [0,T]$,
\begin{equation} \label{A.40}
\widetilde{C}_l^{-1} \omega(0) \le \omega(t) \le \widetilde{C}_l \omega(0). 
\end{equation}
Then 
\begin{equation}
\left| \nabla^k \Rm \right|(x,t) \le \frac{\widehat{C}_l}{r^{k+2}}
\end{equation} 
for all $k \le l$, $x \in B_0 \left(
p, \frac18 r \right)$ and $t \in [0, T]$.
\end{proposition}
\begin{proof}
After rescaling, we can assume that $r=1$.
Proposition \ref{A.1} gives a uniform curvature bound on
$B_0 \left( p, \frac14 \right) \times [0, T]$.
For any $k \ge 1$ and $\epsilon > 0$, Shi's local derivative estimate
implies a bound
\begin{equation}
\left| \nabla^k \Rm \right|(x,t) \le \widehat{C}^{\prime \prime}_k
\end{equation}
on $B_0 \left(
p, \frac18 \right) \times [\epsilon, T]$, where
$\widehat{C}^{\prime \prime}_k$ depends on $n$, $\epsilon$ and
the parameter $C_2$ in the conclusion of Proposition \ref{A.1}.
Given $\alpha > 0$, the local derivative estimate in
\cite[Appendix D]{Kleiner-Lott} gives a bound
\begin{equation}
\left| \nabla^k \Rm \right|(x,t) \le \widehat{C}^{\prime}_l
\end{equation}
on $B_0 \left(
p, \frac18 \right) \times [0, \alpha]$, where $1 \le k \le l$ and
$\widehat{C}^{\prime}_l$ depends on $n$, $\alpha$, $C_2$ and
$\widetilde{C}_l$. Taking $\epsilon < \alpha$, the proposition follows.
\end{proof}

\section{Power law decay of curvature and Ricci flow} \label{powerapp}

In this section we give sufficient conditions for Ricci flow to
preserve a power law decay of curvature.  In Subsection \ref{appsub1}
we show that this is true under a technical condition, which will
be satisfied in the cases of interest. 
The proof is along the lines of Dai-Ma \cite{Dai-Ma}.
In Subsection \ref{appsub2}
we show it is always true for the K\"ahler-Ricci flow.

We remark that Hamilton showed in \cite[Theorem 18.2]{Hamilton3} 
that Ricci flow preserves
the property that the curvature decays to zero at spatial infinity.

\subsection{Power law decay in Ricci flow} \label{appsub1}

Suppose that $(X, x_0, g_X)$ is a complete pointed $n$-dimensional 
Riemannian manifold with $|\Rm| \le k_0$.  
From \cite[Lemma 12.30]{Chowetal2},
there are $C = C(n,k_0) < \infty$ and a function $\phi \in C^\infty(X)$
so that
\begin{align}
C^{-1} (d(x, x_0) + 1) \le & \phi(x) \le C (d(x, x_0) + 1), \\
|d \phi|_{g_X} & \le C, \notag \\
\Hess_{g_X}(\phi) & \le C g_X. \notag
\end{align}
Note that there is an upper bound on $\Hess(\phi)$ but generally
not a lower bound.

Now let $(X, g_X(\cdot))$ be a Ricci flow defined on a time interval
$[0, T)$, with complete time slices and bounded curvature on
compact time intervals.
Let $d_t(\cdot, \cdot)$ denote the time-$t$ distance function.
For any $T^\prime \in [0, T)$, the identity map from 
$(X, g_X(0))$ to $(X, g_X(t))$ is uniformly biLipschitz for
$t \in [0, T^\prime]$.

Let $\phi$ be a distance-like function as above, relative to the
time-zero metric $g_X(0)$. From \cite[Lemma 12.5]{Chowetal2},
given $T^\prime \in [0, T)$, there is some $C_{T^\prime} < \infty$
so that for all $t \in [0, T^\prime]$, we have
\begin{align} \label{4.1.5}
C_{T^\prime}^{-1} (d_t(x, x_0) + 1) \le & \phi(x) \le 
C_{T^\prime} (d_t(x, x_0) + 1), \\
|d \phi|_{g_X(t)} & \le C_{T^\prime}, \notag \\
\Hess_{g_X(t)}(\phi) & \le C_{T^\prime} g_X. \notag
\end{align}
In particular, the notion of power law decay is the same
as measured with $\phi$, $d_0$ or $d_t$. For future convenience,
we will only state results in terms of $d_0$.

\begin{lemma}
If there is some $B_0 \in \R$ so that at time zero, we have
\begin{equation} \label{4.3.5}
\phi^{-1} \Hess_{g_x(0)}(\phi) \ge B_0 g_X(0),
\end{equation}
then for
any $T^\prime \in [0, T)$, there is some $B_{T^\prime} \in \R$
so that for all $t \in [0, T^\prime]$, 
\begin{equation}
\phi^{-1} \Hess_{g_X(t)}(\phi) \ge B_{T^\prime} g_X(t).
\end{equation}
\end{lemma}
\begin{proof}
By the same argument as in \cite[Part (iii) of proof of
Lemma 12.5]{Chowetal2}, 
there is some $C_{T^\prime} < \infty$ so that
for all $t \in [0, T^\prime]$,
\begin{equation}
\Hess_{g_X(t)}(\phi) - \Hess_{g_X(0)}(\phi) \ge -
C_{T^\prime} \cdot g_X(0).
\end{equation}
The lemma follows.
\end{proof}

\begin{proposition} \label{4.1}  (c.f. \cite[Theorem 4]{Dai-Ma})
Let $(X, g(\cdot))$ be a Ricci flow solution on a connected manifold
$X$, defined for $t \in [0, T]$, with uniformly bounded curvature and 
complete time slices.
Let $F(x,t)$ and $u(x,t)$ be smooth bounded functions, with $u$
nonnegative. 
Given $C < \infty$, suppose that
\begin{equation} \label{4.2}
\left( \frac{\partial}{\partial t} - \triangle_{g(t)} \right) u \le 
F \sqrt{u} + Cu. 
\end{equation}
Let $x_0 \in X$ be a basepoint.
For some $\beta > 0$, 
suppose that as $x \rightarrow \infty$, $F(x,t) = 
O \left( d_0(x,x_0)^{-\frac{\beta}{2}} \right)$, uniformly in $t$.
Suppose that  $u(x,0) = 
O \left( d_0(x,x_0)^{- \beta} \right)$.  
Suppose that the time-zero distance-like function
$\phi$ can be chosen to satisfy (\ref{4.3.5}) for some $B_0 \in \R$.
Then we have $u(x,t)=O \left( d_0(x,x_0)^{- \beta} \right)$ 
uniformly in $t \in [0, T]$.
\end{proposition}
\begin{proof}
One calculates that
\begin{align} \label{4.3}
& \left( \frac{\partial}{\partial t} - \triangle_{g(t)} \right) 
(\phi^\beta u) \le \\
& \left( - \beta \frac{\triangle \phi}{\phi} + \beta(\beta+1) 
\frac{|\nabla \phi|^2}{\phi^2} \right)  \phi^{\beta} u 
- 2 \beta 
\left\langle \frac{\nabla \phi}{\phi}, \nabla (\phi^\beta u) \right\rangle +
\phi^{\frac{\beta}{2}} F \sqrt{\phi^\beta u} + C\phi^\beta u. \notag
\end{align}
By the weak maximum principle
\cite[Theorem 12.14]{Chowetal2}, $\phi^\beta u$ is 
uniformly bounded above for all $t \in [0,T]$.
This proves the proposition.
\end{proof}

\begin{proposition} \label{4.5}  (c.f. \cite[Theorem 1A]{Dai-Ma})
Let $(X, g(\cdot))$ be a Ricci flow solution on a connected manifold
$X$, defined for $t \in [0, T]$, with uniformly bounded curvature and 
complete time slices.

Let $x_0 \in X$ be a basepoint.
For some $\alpha > 0$, 
suppose that as $x \rightarrow \infty$, 
$|\Rm(x,0)| = 
O \left( d_0(x,x_0)^{- 2\alpha} \right)$.  
Suppose that the time-zero distance-like function
$\phi$ can be chosen to satisfy (\ref{4.3.5}) for some $B_0 \in \R$.
Then we have $|\Rm(x,t)| = O \left( d_0(x,x_0)^{- 2\alpha} \right)$ 
uniformly in $t \in [0, T]$.
\end{proposition}
\begin{proof}
From
\cite[(6.1)]{Chowetal},
\begin{equation} \label{4.6}
\left( \frac{\partial}{\partial t} - \triangle_{g(t)} \right)
 |\Rm|^2 \le 16 |\Rm|  \cdot  |\Rm|^2.
\end{equation}
By assumption, there is some $K < \infty$ so that
$|\Rm(x,t)|\le K$ for all $(x,t)$.
Put $u(x,t) = |\Rm(x,t)|^2$. 
Applying Proposition \ref{4.1} with $F = 0$ and $C = 16K$, the claim follows.
\end{proof}

\begin{proposition} \label{4.7}
Under the hypotheses of Proposition \ref{4.5}, suppose that
$|\nabla^k \Rm|(x,0) = O \left(
d_0(x,x_0)^{-(k+2) \alpha} \right)$, for all
$0 \le k \le l$. Then for all $0 \le k \le l$ and $t$, we have
$|\nabla^k \Rm|(x,t) = O \left (d_0(x,x_0)^{-(k+2) \alpha} \right)$,
uniformly in $t \in [0, T]$.
\end{proposition}
\begin{proof}
Given $1 \le k \le l$,
put $u(x,t) = |\nabla^k \Rm(x,t)|^2$.

By assumption, $u(x,0) = O \left( 
d_0(x,x_0)^{- 2(k+2) \alpha} \right)$. 
From
\cite[(6.24)]{Chowetal},
\begin{equation} \label{4.8}
\left( \frac{\partial}{\partial t} - \triangle_{g(t)} \right)
u \le c(n) \sum_{l=1}^{k-1} |\nabla^l \Rm| \cdot |\nabla^{k-l} \Rm| \sqrt{u} +
c(n) |\Rm| u.
\end{equation}
By induction, we can assume that 
\begin{equation} \label{4.9}
|\nabla^l \Rm| \cdot |\nabla^{k-l} \Rm| = O \left( 
d_0(x,x_0)^{- (k+4) \alpha}
\right),
\end{equation}
uniformly in $t \in [0, T]$.
The claim now follows from Proposition \ref{4.1}.
\end{proof}

\subsection{Power law decay in K\"ahler-Ricci flow} \label{appsub2}

\begin{proposition} \label{B.15}
Let $X$ be a complex manifold.  Let $(X, g(\cdot))$ be a
K\"ahler-Ricci flow on $X$, defined for $t \in [0, T)$,
with complete time slices and
bounded curvature on compact time intervals. Let $x_0 \in X$
be a basepoint. For some $\alpha \in (0,1]$ and $l \ge 1$, suppose
that $|\nabla^k \Rm|(x,0) = O \left(
d_0(x,x_0)^{-(k+2) \alpha} \right)$, for all
$0 \le k \le l$. Then for all $0 \le k \le l$ and $t$, we have
$|\nabla^k \Rm|(x,t) = O \left (d_0(x,x_0)^{-(k+2) \alpha} \right)$,
uniformly in $t \in [0, T]$.
\end{proposition}
\begin{proof}
Given $p \in X - B_0(x_0, 1)$, put $r = \frac12 d(p,x_0)^\alpha$. 
By assumption, there is some $\widetilde{C}_l < \infty$ so that
the hypotheses of Proposition \ref{A.38} are satisfied for all such $p$.
(Hypothesis \ref{A.40} is satisfied because the bounded curvature
assumption implies biLipschitzness on a finite time interval.)
Then Proposition \ref{A.38}, applied at the center of the ball
$B_0 \left(p, \frac18 r \right)$, implies the proposition.
\end{proof}

\bigskip
\footnotesize
\noindent\textit{Acknowledgments.}

We thank Albert Chau, Xianzhe Dai and Hans-Joachim Hein for
helpful discussions.
We also thank Hans-Joachim for pointing out
a mistake in an earlier version of the paper.

The research of the first author was partially supported
by NSF grant DMS-1207654 and a Simons Fellowship.
The research of the second author was partially supported
by ARC Discovery Project DP110102654.


\begin{thebibliography}{$$}

\normalsize
\baselineskip=17pt

\bibitem{Cabezas-Rivas-Wilking} Cabezas-Rivas, E., Wilking, B.:
How to produce a Ricci flow via Cheeger-Gromoll exhaustion.
J. Eur. Math Soc. \textbf{17}, 3153-3194 (2015) 

\bibitem{chau}  Chau, A.: Convergence of the K\"ahler-Einstein flow
on noncompact K\"ahler manifolds. J. Diff. Geom. \textbf{66},
 211-232 (2004)

\bibitem{chau-tam} Chau, A., Tam, L.-F.:
A $C^0$-estimate for the parabolic Monge-Amp\`ere equation on
complete non-compact K\"ahler manifolds.
Compositio Math. \textbf{146},  259-270 (2010)

\bibitem{chen-zhu} Chen, B.-L., Zhu, X.-P.:
Uniqueness of the Ricci flow on complete noncompact manifolds.
J. Diff. Geom. \textbf{74},  119-154 (2006)

\bibitem{Chodosh-Fong} Chodosh, O., Fong, F.:
Rotational symmetry of conical K\"ahler-Ricci solitons.
Math. Ann. \textbf{364}, 777-792 (2016)

\bibitem{Chowetal} Chow, B., Lu, P., Ni, L.: 
Hamilton's Ricci flow, Amer. Math. Soc., Providence (2006)

\bibitem{Chowetal2}  Chow, B., Chu, S.-C., Glickenstein, D.,
Guenter, C., Isenberg, J., Ivey, T., Knopf, D., Lu, P., Luo, F., Ni, L.:
The Ricci Flow: Techniques and Applications. Part II.
Analytic Aspects, Mathematical Surveys and Monographs 144,
Amer. Math. Soc., Providence (2008)

\bibitem{Dai-Ma} Dai, X., Ma, L.: Mass under the Ricci flow.
Comm. Math. Phys. \textbf{274},  65-80 (2007)

\bibitem{FIK}  Feldman, M., Ilmanen, T., Knopf, D.: 
Rotationally symmetric shrinking and expanding gradient
K\"ahler-Ricci solitons. J. Diff. Geom. \textbf{65},  169-209 (2003)

\bibitem{GH} Gibbons, G., Hawking, S.: 
Gravitational multi-instantons. Phys. Lett. B \textbf{78},  430
(1978)

\bibitem{Hamilton2} Hamilton, R.: A compactness property for
solutions of the Ricci flow. Amer. J. Math. \textbf{117},  545-572 (1995)

\bibitem{Hamilton3} Hamilton, R.: Formation of singularities
in the Ricci flow. In: Surveys in Differential Geometry,
Vol. 2, International Press, Somerville,  7-136 (1995)

\bibitem{Hilaire} Hilaire, C.: Ricci flow on Riemannian groupoids.
arxiv:1411.6058 (2014)

\bibitem{Kleiner-Lott} Kleiner, B., Lott, J.:
Notes on Perelman's papers. Geometry and Topology \textbf{12}, 
2587-2855 (2008)

\bibitem{lz-duke} Lott, J., Zhang, Z.: Ricci flow on quasi-projective 
manifolds. Duke Math. J. \textbf{156}, no. 1, 87--123 (2011)

\bibitem{Lu} Lu, P.: A local curvature bound in Ricci flow.
Geometry and Topology \textbf{14},  1095-1110 (2010)

\bibitem{Perelman} Perelman, G.: The entropy formula for the
Ricci flow and its geometric applications. 
arXiv:0211159 (2002)

\bibitem{rz-advances} Rochon, F., Zhang, Z.: Asymptotics of 
complete K\"ahler metrics of finite volume on quasiprojective manifolds. 
Adv. Math. \textbf{235}, 2892--2952 (2012)

\bibitem{Schulze-Simon} Schulze, F., Simon, M.:
Expanding solitons with non-negative curvature operator coming
out of cones. Math. Z. \textbf{275},  625-639 (2013)

\bibitem{sherman-weinkove} Sherman, M., Weinkove, B.: Interior 
derivative estimates for the K\"ahler-Ricci flow.
Pac. J. Math. \textbf{257}, 491-501 (2012)

\bibitem{song-weinkove} Song, J., Weinkove, B.: Introduction to the
K\"ahler-Ricci flow. In: An Introduction to the K\"ahler-Ricci Flow,
Lecture Notes in Math. 2086, Springer, New York, 89-188 (2013)

\bibitem{tian-yau} Tian, G.,  Yau, S.-T.:
Existence of K\"ahler-Einstein metrics on complete K\"ahler manifolds
and their applications to algebraic geometry. In:
Mathematical Aspects of String
Theory (San Diego, Calif., 1986), Adv. Ser. Math. Phys. 1, 
World Sci. Publishing, Singapore, 574-628 (1987)

\bibitem{tian-yau2} Tian, G., Yau, S-T.:
Complete K\"ahler manifolds with zero Ricci curvature I.
J. Amer. Math. Soc. \textbf{3},  579-609 (1990)

\bibitem{tian-yau3} Tian, G., Yau, S.-T.:
Complete K\"ahler manifolds with zero Ricci curvature II.
Invent. Math. \textbf{106},  27-60 (1991)

\bibitem{t-znote} Tian, G., Zhang, Z.: On the 
K\"ahler-Ricci flow on projective manifolds of general 
type. Chinese Annals of Math. - Series B \textbf{27}, 
 179-192 (2006)   

\end{thebibliography}
\end{document}